\documentclass[11pt]{amsart}
\usepackage{amssymb,amsmath,amsthm}
\usepackage{tikz}
\usepackage{enumitem}
 \usepackage{tikz-cd}
 \usepackage{mathtools}
\usetikzlibrary{positioning}
\usetikzlibrary{matrix,arrows,decorations.pathmorphing, shapes.geometric}
\tikzset{>=latex}
\usepackage{verbatim}
\usepackage{mathrsfs}
\usepackage[margin=1.25in]{geometry}

\usepackage{xcolor}

\newtheorem{theorem}{Theorem}[section]
\newtheorem{prop}[theorem]{Proposition}
\newtheorem{lemma}[theorem]{Lemma}
\newtheorem{corollary}[theorem]{Corollary}

\newtheorem{definition}[theorem]{Definition}
\newtheorem{question}[theorem]{Question}

\theoremstyle{definition} 

\newtheorem{remark}[theorem]{Remark}

\numberwithin{equation}{section} 
\numberwithin{theorem}{section}
\numberwithin{figure}{section}

\newcommand{\R}{\mathbb{R}}

\newcommand{\sn}{\operatorname{sn}}
\newcommand{\ct}{\operatorname{ct}}
\newcommand{\cs}{\operatorname{cs}}
\newcommand{\Acal}{\mathcal{A}}

\newcommand{\Ecal}{\mathcal{E}}
\newcommand{\Gcal}{\mathcal{G}}
\newcommand{\Hcal}{\mathcal{H}}
\newcommand{\Jcal}{\mathcal{J}}
\newcommand{\Kcal}{\text{\Tiny $\mathcal{K}$}}

\newcommand{\Ucal}{\mathcal{U}}

\newcommand{\rootg}{\sqrt{g}}

\begin{document}
		
\author[J.A.~Hoisington]{Joseph Ansel Hoisington}\address{Department of Mathematics, University of Georgia, Athens, GA 30602 USA}\email{jhoisington@uga.edu} 
		
\title[Hypersurfaces, Geodesics and Isoperimetric Inequalities]{Hypersurfaces, Geodesics and Isoperimetric Inequalities in Cartan-Hadamard Manifolds}
		
		
\begin{abstract}
We prove an inequality for submanifolds of Cartan-Hadamard manifolds, which relates the geometry of a submanifold to the measure of the geodesics in the ambient space which it intersects.  For hypersurfaces, this gives an extension of Banchoff and Pohl's isoperimetric inequality to spaces of non-positive curvature.  We also prove a modified version of Croke's isoperimetric inequality for hypersurfaces immersed in Cartan-Hadamard manifolds and a sharp, quantitative version of an isoperimetric inequality of Yau in spaces of negative curvature.  We discuss the relationship between these results, and we develop several facts about the spaces of geodesics in Cartan-Hadamards manifold that may be of independent interest. 
\end{abstract}

\maketitle
	

\section{Introduction}
\label{intro}  

The isoperimetric inequality states that among all domains with the same perimeter in $n$-dimensional Euclidean space, the ball is the unique domain with maximal volume.  This is usually expressed as an inequality between the volume $|\Omega|$ and perimeter $|\partial \Omega|$ of a domain $\Omega$ in $\R^{n}$:  

\begin{equation}
\label{iso_ineq}
\displaystyle |\partial \Omega|^{n} \geq \ \scriptstyle  \frac{\Sigma^{n}_{n-1}}{B^{n-1}_{n}} \displaystyle |\Omega|^{n-1}, \bigskip 
\end{equation}
with equality precisely if $\Omega$ is a ball.  In the constant in (\ref{iso_ineq}), $\Sigma_{n-1}$ is the surface area of the unit $(n-1)$-sphere and $B_{n}$ is the volume of the unit $n$-ball.  We adopt this notation throughout the paper. \\ 

There is a long-standing conjecture that the inequality (\ref{iso_ineq}) is true for any domain in a complete, simply-connected Riemannian $n$-manifold with non-positive curvature.  These spaces are known as Cartan-Hadamard manifolds and the conjecture has become known as the Cartan-Hadamard conjecture.  At present, the conjecture is known to be true in dimensions $2$, $3$ and $4$.  The $2$-dimensional case was first proven by Weil \cite{We} and later independently by Beckenbach-Rad\'o \cite{BR}, the $3$-dimensional case was proved by Kleiner \cite{Kl} and the $4$-dimensional case was proved by Croke \cite{Cr}.  Kleiner's proof of the $3$-dimensional result established a stronger version of the conjecture in dimension $3$, which we will discuss below.  Croke's proof of the $4$-dimensional result also established non-sharp isoperimetric inequalities in Cartan-Hadamard manifolds of all dimensions, which strengthened earlier non-sharp inequalities of Hoffman-Spruck \cite{HS} and Croke \cite{Cr2}. \\ 

Our first main result is an inequality, based on the isoperimetric inequality, for hypersurfaces immersed in Cartan-Hadamard manifolds:  

\begin{theorem}
\label{main_thm_1}
Let $f: M^{n-1} \rightarrow \Hcal^{n}$ be an immersion of a closed, oriented hypersurface $M$ in a Cartan-Hadamard manifold $\Hcal$.  For $x,y \in M$, let $r(x,y)$ be the distance between $f(x)$ and $f(y)$ in $\Hcal$, and for $p \in \Hcal$, let $w(M,p)$ be the winding number of $M$ about $p$.  Then: 

\begin{equation}
\label{main_thm_1_eqn}
n \Sigma_{n-1} \displaystyle \int\limits_{\Hcal} w(M,p)^{2} \ dVol_{\Hcal} \leq \iint\limits_{M \times M} \frac{1}{r^{n-2}} \ dVol_{M \times M}. \medskip
\end{equation} 

Equality holds if and only if $f(M)$ is the boundary, possibly with multiplicity, of a domain $\Omega$ isometric to a ball in $\R^{n}$. \end{theorem}

For an embedded hypersurface, the expression $\int_{\Hcal} w(M,p)^{2} dVol_{\Hcal}$ is equal to the volume of the domain enclosed by $M$.  For hypersurfaces in Euclidean space, this expression appears in a classical inequality of Banchoff and Pohl \cite{BP}.  As we will explain below, their result gives a far-reaching generalization of the isoperimetric inequality in the plane.  Theorem \ref{main_thm_1} gives an extension of the Banchoff-Pohl inequality to Cartan-Hadamard manifolds. \\  

We will also prove the following inequality, based on Croke's isoperimetric inequality in \cite{Cr}.  The invariant $\Acal(f)$ in this result, which we will define below, is essentially an average of the number of times that geodesics in $\Hcal$ occur as chords of a hypersurface $M$.   

\begin{theorem}
\label{generalization}
Let $f:M^{n-1} \rightarrow \Hcal^{n}$ be an immersion of a closed, oriented hypersurface $M$ in a Cartan-Hadamard manifold $\Hcal$.  Let $w(M,p)$ be the winding number of $M$ about $p$ and $\Acal(f)$ as in Definition \ref{average} below.  Then: 

\begin{equation}
\label{generalization_eqn_1} 
\displaystyle C_{n} \int\limits_{\Hcal} w(M,p)^{2} dVol_{\Hcal} \leq Vol(M)^{\frac{n}{n-1}} \Acal(f)^{\frac{n-2}{n-1}}, \medskip 
\end{equation}
where $C_{n}$ is a constant depending only on $n$, which will be defined in (\ref{generalization_constant}).  Equality holds if and only if $n = 2$ or $4$ and $f(M)$ is the boundary, possibly with multiplicity, of a domain $\Omega$ isometric to a ball in $\R^{n}$. \end{theorem}

The constant $C_{n}$ in Theorem \ref{generalization} is the same as in Croke's inequality, which says that for a domain $\Omega$ in a Cartan-Hadamard manifold $\Hcal$,   

\begin{equation}
\label{croke_inequality}
\displaystyle C_{n} |\Omega| \leq |\partial \Omega|^{\frac{n}{n-1}}, \bigskip 
\end{equation}
with equality if and only if $n = 2$ or $4$ and $\Omega$ is isometric to a ball in $\R^{n}$. \\  

One of the key steps in the proofs of Theorems \ref{main_thm_1} and \ref{generalization}, and of the Banchoff-Pohl inequality, is a study of the secant mapping of the submanifold $M$.  This mapping sends a pair of points $(x,y) \in M \times M$ with $f(x) \neq f(y)$ to the geodesic in $\Hcal$ through $f(x)$ and $f(y)$.  The set of oriented geodesics in an $n$-dimensional Cartan-Hadamard manifold is canonically a $(2n-2)$-dimensional symplectic manifold, diffeomorphic to the tangent bundle of the $(n-1)$-sphere.  In Section \ref{secants}, we will describe this space of geodesics and the secant mapping and establish some of their properties.  In the process of proving Theorem \ref{main_thm_1}, we will prove:  

\begin{theorem}
\label{main_thm_2}
Let $f: M^{m} \rightarrow \Hcal^{n}$ be an immersion of a closed, oriented $m$-manifold $M$ in an $n$-dimensional Cartan-Hadamard manifold $\Hcal$.  Let $\Gcal$ be the space of oriented geodesics in $\Hcal$ and $dI$ the symplectic form on $\Gcal$, as in Definitions \ref{space_of_G} and \ref{G_form}, and let $l:M \times M \rightarrow \Gcal$ be the secant mapping as in Definition \ref{secant_mapping}.  Let $r(x,y)$ be the chordal distance between $x,y \in M$ as in Theorem \ref{main_thm_1}.  Then:  
	
\begin{equation}
\label{main_thm_2_eqn}
\displaystyle (m+1) D_{m} \iint\limits_{M \times M} r \ l^*(dI)^{m} \leq \iint\limits_{M \times M} \frac{1}{r^{m-1}} dVol_{M \times M}, \medskip 
\end{equation}
where $D_{m} = \frac{(-1)^{(\binom{m+1}{2}+1)}}{m!}$.  Equality holds if and only if $f(M)$ is the boundary, possibly with multiplicity, of an embedded, totally geodesic submanifold $\mathcal{D}^{m+1}$ of $\Hcal$ which is isometric to a ball in $(m+1)$-dimensional Euclidean space. \end{theorem}

A stronger version of the Cartan-Hadamard conjecture states that if the sectional curvature of a Cartan-Hadamard manifold $\Hcal^{n}$ is bounded above by a negative constant $\mathcal{K}$, then domains in $\Hcal$ satisfy the isoperimetric inequality of the $n$-dimensional hyperbolic space of curvature $\mathcal{K}$.  Equality should again hold only for domains isometric to a geodesic ball in hyperbolic space.  This is sometimes known as the generalized Cartan-Hadamard conjecture and is known to be true in dimensions $2$ and $3$.  The $2$-dimensional case was first proved by Bol \cite{Bo} and the $3$-dimensional case was proved by Kleiner along with the $3$-dimensional case of the original conjecture in \cite{Kl}.  In the setting of the generalized Cartan-Hadamard conjecture, we will prove:  
\medskip 

\begin{flushleft}
{\bf Theorem \ref{main_thm_1}.B} {\em Let $f: M^{n-1} \rightarrow \Hcal^{n}$ be an immersion of a closed, oriented hypersurface $M$ in a Cartan-Hadamard manifold $\Hcal$, and suppose the sectional curvature of $\Hcal$ is bounded above by $\mathcal{K} < 0$.  Let $w(M,p)$ and $r(x,y)$ be as above, and for $(x,y) \in M \times M$, let $\sigma_{1}, \sigma_{2} \in [0,\frac{\pi}{2}]$ be the angle made by the geodesic chord from $f(x)$ to $f(y)$ with $M$ at $x$ and $y$ respectively.  Let $\sn_{\Kcal}(r) = \frac{1}{\sqrt{|\Kcal|}}\sinh(\sqrt{|\Kcal|}r)$.  Then:}
\end{flushleft}

\begin{equation}
\label{main_thm_1_B_eqn}
\displaystyle n \Sigma_{n-1} \int\limits_{\Hcal} w(M,p)^{2} \ dVol_{\Hcal} \leq  \iint\limits_{M \times M} \frac{\sn_{\Kcal}(r) - \sin(\sigma_{1})\sin(\sigma_{2})(\sn_{\Kcal}(r) - r)}{\sn_{\Kcal}^{n-1}(r)} dVol_{M \times M}. \medskip 
\end{equation}

{\em Equality holds if and only if $f(M)$ is the boundary, possibly with multiplicity, of a domain $\Omega$ isometric to a ball in an $n$-dimensional hyperbolic space of curvature $\mathcal{K}$.} 

\begin{flushleft}
{\bf Theorem \ref{main_thm_2}.B} {\em Let $f: M^{m} \rightarrow \Hcal^{n}$ be an immersion of a closed, oriented $m$-manifold $M$ in an $n$-dimensional Cartan-Hadamard manifold $\Hcal$, with the sectional curvature of $\Hcal$ bounded above by $\mathcal{K} < 0$.  Let $\Gcal$, $dI$, $l:M \times M \rightarrow \Gcal$, $r(x,y)$ and $\sn_{\Kcal}(r)$ be as above.  Then:}
\end{flushleft}

\begin{equation}
\label{main_thm_2_B_eqn}
\displaystyle D_{m} \iint\limits_{M \times M} \left( \sn_{\Kcal}(r) + mr \right) \ l^*(dI)^{m} \leq \iint\limits_{M \times M} \frac{1}{\\sn_{\Kcal}^{m-1}(r)}  dVol_{M \times M}, \medskip 
\end{equation}
{\em where $D_{m} = \frac{(-1)^{(\binom{m+1}{2}+1)}}{m!}$ as above.  Equality holds if and only if $f(M)$ is the boundary, possibly with multiplicity, of an embedded, totally geodesic submanifold $\mathcal{D}_{\Kcal}^{m+1}$ of $\Hcal$ which is isometric to a ball in an $(m+1)$-dimensional hyperbolic space of curvature $\mathcal{K}$.} 
\medskip

We will also prove the following inequality, related to Theorem \ref{main_thm_1}.B.  As we will explain, this result gives a sharp quantitative version of an isoperimetric inequality of Yau:  

\begin{theorem}
\label{yau_inequality}
Let $f:M^{n-1} \rightarrow \Hcal^{n}$ be an immersion of a closed, oriented hypersurface $M$ in a Cartan-Hadamard manifold $\Hcal$, and suppose the sectional curvature of $\Hcal$ is bounded above by a negative constant $\mathcal{K}$.  Let $r$ be the chordal distance function on $M \times M$ as above and $\nabla r$ the gradient of $r$ on $M \times M$.  Let $w(M,p)$ be the winding number of $M$ about a point $p$ in $\Hcal$.  Then:  

\begin{equation}
\label{yau_inequality_eqn}
\displaystyle Vol(M)^{2} - \left( (n-1) \sqrt{|\Kcal|} \int\limits_{\Hcal} |w(M,p)| dVol_{\Hcal} \right)^{2} \geq \iint\limits_{M \times M} \Psi_{\Kcal}^{n}(r,|\nabla r|) dVol_{M \times M} > 0, \medskip 
\end{equation} 
where $\Psi_{\Kcal}^{n}(r,|\nabla r|)$ is as defined in (\ref{yau_inequality_function}). \\ 

The equality $Vol(M)^{2} - \left( (n-1) \sqrt{|\Kcal|} \int_{\Hcal} |w(M,p)| dVol_{\Hcal} \right)^{2} = \iint_{M \times M} \Psi_{\Kcal}^{n}(r,|\nabla r|) dVol_{M \times M}$ holds if and only if $f(M)$ is the boundary, possibly with multiplicity, of a domain isometric to a ball in an $n$-dimensional hyperbolic space of curvature $\mathcal{K}$. \end{theorem}

As with the expression $\int_{\Hcal} w(M,p)^{2} dVol_{\Hcal}$ in Theorems \ref{main_thm_1} and \ref{generalization}, for an embedded hypersurface $\int_{\Hcal}|w(M,p)|dp$ is the volume of the domain enclosed by $M$.  Theorem \ref{yau_inequality} therefore gives a quantitative version of the following inequality of Yau: 

\begin{theorem}[Yau, \cite{Yau}]
\label{yau_thm}
Let $\Omega$ be a compact domain in an $n$-dimensional Cartan-Hadamard manifold $\Hcal^{n}$ whose sectional curvature is bounded above by a negative constant $\mathcal{K}$.  Then: 

\begin{equation}
\label{yau_thm_eqn}
\displaystyle |\partial \Omega| > (n-1)\sqrt{|\Kcal|} |\Omega|. \medskip 
\end{equation}
\end{theorem}

For a curve $M$ in a surface of non-positive curvature, Theorem \ref{main_thm_1}.B has a geometric interpretation which is related to Theorem \ref{yau_inequality}.  This and the $2$-dimensional case of Theorem \ref{main_thm_1} have also been proven by Howard \cite{Ho}.  We record these results in the following: 

\begin{theorem}[see also \cite{Ho}]
\label{Howard_cor}
Let $M$ be a closed curve of length $L$, not necessarily simple, in a complete, simply-connected surface $\Hcal^{2}$, with curvature of $\Hcal^{2}$ bounded above by $\mathcal{K} \leq 0$.  Then: 

\begin{equation}
\label{Howard_cor_eqn}
\displaystyle 4\pi \int\limits_{\Hcal} w(M,p)^{2} dA_{\Hcal} + |\mathcal{K}|\left( \int\limits_{\Hcal} |w(M,p)| \ dA_{\Hcal} \right)^{2} \leq L^{2}. \medskip 
\end{equation}
	
Equality holds if and only if $M$ is the boundary, possibly with multiplicity as above, of a domain isometric to a disk with constant curvature $\mathcal{K}$. \end{theorem}  

This gives an extension of the $2$-dimensional generalized Cartan-Hadamard conjecture, which says that for a domain $\Omega$ in a surface $\Hcal^{2}$ as above, whose area is $A$ and perimeter has length $L$, 

\begin{equation}
\label{2-dim_CH_ineq}
\displaystyle 4\pi A + |\mathcal{K}| A^{2} \leq L^{2}. \bigskip
\end{equation}

Our proof of Theorem \ref{Howard_cor} is in a sense the converse of Howard's: Howard proves Theorem \ref{Howard_cor} by using the classical inequality (\ref{2-dim_CH_ineq}) to prove sharp Sobolev inequalities for functions of bounded variation in Cartan-Hadamard surfaces, which give (\ref{Howard_cor_eqn}) when applied to the function $w(M,p)$.  We will prove Theorem \ref{Howard_cor} directly, with the isoperimetric inequality (\ref{2-dim_CH_ineq}) as a special case.  We note that Howard's results cover a more general situation than in Theorem \ref{Howard_cor} -- in particular, his results apply to surfaces with a positive upper curvature bound, subject to some additional hypotheses which are necessary in that setting.  \\ 

In \cite{BP}, building on earlier work of Pohl \cite{Pohl}, Banchoff and Pohl proved the following generalization of the classical isoperimetric inequality in the plane.  Much of our initial motivation for the results in this paper came from a desire to extend the Banchoff-Pohl inequality to spaces of non-positive curvature:  

\begin{theorem}[The Banchoff-Pohl Inequality, \cite{BP}]
\label{bp_ineq}
	
Let $f: M^{m} \rightarrow \R^{n}$ be an immersion of a closed, oriented manifold $M$ in Euclidean space.  Let ${\bf H}_{n-m-1}$ be the space of affine $(n-m-1)$-planes in $\R^{n}$ and $dE$ a measure on ${\bf H}_{n-m-1}$ which is invariant under the action of the Euclidean isometry group.  For $x, y \in M$, let $r(x,y)$ be the distance between $f(x)$ and $f(y)$ in $\R^{n}$, and for $E \in {\bf H}_{n-m-1}$, let $\lambda(M,E)$ be the linking number of $M$ about $E$.  Then:  
	
\begin{equation}
\label{bp_ineq_eqn}
\displaystyle (m+1) \Sigma_{m} K_{n,m} \int\limits_{{\bf H}_{n-m-1}} \lambda^{2}(M,E) \ dE \leq \iint\limits_{M \times M} \frac{1}{r^{m-1}} \ dVol_{M \times M}, \medskip 
\end{equation}
where $K_{n,m}$ is a constant depending only on $n$ and $m$.  Equality holds precisely if $M$ is a round sphere, or several coincident round spheres with the same orientation or, if $m=1$, one or several coincident circles, each traversed in the same direction some number of times. \end{theorem}  

When $M$ is a simple closed curve of length $L$ in the plane, enclosing a domain of area $A$, the Banchoff-Pohl inequality gives the classical isoperimetric inequality $L^{2} \geq 4\pi A$.  When the Cartan-Hadamard manifold $\Hcal$ in Theorems \ref{main_thm_1} and \ref{main_thm_2} is $n$-dimensional Euclidean space, it follows from Banchoff and Pohl's work that the left-hand side of (\ref{bp_ineq_eqn}) is equal to the integral $D_{m} \iint_{M \times M} r \ l^{*}(dI)^{m}$ which is bounded above in Theorem \ref{main_thm_2}, up to a positive constant depending only on $n$ and $m$, so that Theorem \ref{bp_ineq} becomes a special case of Theorem \ref{main_thm_2}.  So, although the result usually does not have an interpretation in terms of linking numbers as in Theorem \ref{bp_ineq}, Theorem \ref{main_thm_2} gives a sharp extension of the Banchoff-Pohl inequality to Cartan-Hadamard manifolds in complete generality.  Moreover, when the ambient space $\Hcal$ is a hyperbolic space, Theorem \ref{main_thm_2}.B implies a theorem of Gysin \cite{Gy} which gives a sharp extension of the Banchoff-Pohl inequality to spaces of constant curvature.  Theorems \ref{main_thm_1} and \ref{main_thm_2} therefore show that, although the Cartan-Hadamard and generalized Cartan-Hadamard conjectures are currently open in all dimensions $n \geq 5$, the equivalent conjectures for the Banchoff-Pohl inequality are true. \\ 

There are several questions related to the results of this paper and the Banchoff-Pohl inequality that we believe it may be interesting to study in future work.  One of the most natural seems to be whether, for a codimension $2$ submanifold $M^{n-2}$ of a Cartan-Hadamard manifold $\Hcal^{n}$, the Banchoff-Pohl inequality holds for the linking number of $M$ about geodesics in $\Hcal$ in the following sense: 

\begin{question}
\label{linking_question}
Let $f:M^{n-2} \rightarrow \Hcal^{n}$ be an immersion of a closed, oriented manifold $M$ of codimension $2$ in a Cartan-Hadamard manifold $\mathcal{H}$.  Let $\Gcal$ be the space of geodesics in $\Hcal$, as in Definition \ref{space_of_G}, and let $dVol_{\Gcal}$ be the canonical measure on $\Gcal$, as in Definition \ref{G_measure}.  Let $\lambda(M,\gamma)$ be the linking number of $M$ about an oriented geodesic $\gamma$ in $\Hcal$ and let $K_{n,n-2}$ be the constant in Theorem \ref{bp_ineq}.  Is it the case that:  

\begin{equation}
\label{conjectured_ienquality}
\displaystyle (n-1) \Sigma_{n-2} K_{n,n-2} \int\limits_{\Gcal} \lambda^{2}(M,\gamma) dVol_{\Gcal} \leq \iint\limits_{M \times M} \frac{1}{r^{n-3}} dVol_{M \times M}, \medskip 
\end{equation}
with equality if and only if $f(M)$ is the boundary of an embedded, totally geodesic submanifold isometric to a disk in $\R^{n-1}$? \end{question}  

More precisely, when $M$ is an immersed, codimension $2$ submanifold of a Cartan-Hadamard manifold $\Hcal$, it would be natural to conjecture that $(n-1) \Sigma_{n-2} K_{n,n-2} \int_{\Gcal} \lambda^{2}(M,\gamma) dVol_{\Gcal}$ is equal to the invariant $D_{m}\iint_{M \times M} r \ l^{*}(dI)^{m}$ which is bounded above in Theorem \ref{main_thm_2}.  If $M$ bounds an embedded, orientable hypersurface, this should follow from an adaptation of Pohl's and Banchoff-Pohl's arguments in \cite{Pohl, BP}.  In particular, for $3$-dimensional Cartan-Hadamard manifolds $\Hcal^{3}$, this should follow immediately for embedded closed curves $M^{1}$ in $\Hcal$ via the existence of a Seifert surface for $M$, and then for all closed curves in $\Hcal^{3}$ by the genericity of embeddedness among smooth immersions $f:M^{1} \rightarrow \Hcal^{3}$.  In fact, near the completion of this work, we learned that Teufel has proven the $3$-dimensional case of this result in a study of integral geometry in Riemannian manifolds \cite{Te}.  Teufel also proves a result similar to Theorem \ref{generalization}.  In this result, however, one takes the maximum number of transverse intersection points of a geodesic $\gamma$ in $\Hcal$ with a hypersurface $M$, rather than the average $\Acal(f)$ in Theorem \ref{generalization}. \\ 

In the setting where a codimension $2$ submanifold $M^{n-2}$ of $\Hcal^{n}$ does not bound such a spanning hypersurface, it may still be possible to adapt Banchoff and Pohl's arguments to establish the conjectured identity for $\int_{\Gcal} \lambda^{2}(M,\gamma) dVol_{\Gcal}$ and the inequality (\ref{conjectured_ienquality}).  However there is another approach to this question that we believe might work well and might be interesting in its own right.  We will discuss this in Remark \ref{current_remark}. \\  

We will finish this introduction by briefly giving an outline of the rest of the paper and references to additional background reading about isoperimetric inequalities, the Cartan-Hadamard conjecture and the Banchoff-Pohl inequality: \\  

In Section \ref{secants}, we will define the space of geodesics in a Cartan-Hadamard manifold and the secant space of a closed, oriented manifold.  We will describe the secant mapping which is induced by an immersion of a submanifold $M$ into a Cartan-Hadamard manifold $\Hcal$ and establish some of its basic geometric properties.  To our knowledge, this is the first time the secant mapping has been defined and studied in Cartan-Hadamard manifolds, and we believe this development might be of independent interest. \\ 

In Section \ref{bp_in_ch}, we will prove Theorems \ref{main_thm_1} and \ref{main_thm_2}.  We will show that Theorem \ref{main_thm_1} follows from Theorem \ref{main_thm_2} by of a formula for the invariant $\int_{\Hcal}w(M,p)^{2} dp$ which we will establish in Proposition \ref{hyp_main_thm}.  We will also prove that the upper bound $\iint_{M \times M} \frac{1}{r^{m-1}} dVol_{M\times M}$ in Theorems \ref{main_thm_1} and \ref{main_thm_2} is finite. \\  

In Section \ref{isoperimetric_inequalities}, we will prove Theorems \ref{yau_inequality}, \ref{Howard_cor} and \ref{generalization}.  We note that Theorem \ref{yau_inequality} implies that the Cartan-Hadamard conjecture holds for domains $\Omega$ in Cartan-Hadamard manifolds with a negative upper curvature bound if the volume of $\Omega$ is sufficiently large, cf. \cite{Kl}.  A more detailed study of the estimate in Theorem \ref{yau_inequality} may therefore have some applicability to the Cartan-Hadamard conjecture.  For convex, embedded hypersurfaces $M^{n-1} \subseteq \Hcal^{n}$, Theorem \ref{generalization} coincides with Croke's isoperimetric inequality (\ref{croke_inequality}).  However if $M$ is an embedded hypersurface but is not convex, $A(f) > 1$ and Croke's inequality gives a stronger result than Theorem \ref{generalization}.  In Remark \ref{croke_remark}, we will describe a possible strengthening of Theorem \ref{generalization} which we believe would be a natural goal for future work. \\ 

Osserman's survey article \cite{Os} gives a history of the isoperimetric inequality and its influence in geometry, analysis and physics.  The Cartan-Hadamard conjecture has appeared in work of Aubin \cite{Aub1}, Burago and Zalgaller \cite{BZ1} and Gromov \cite{Gr1}, and \cite{KK1,GS1} give histories of the Cartan-Hadamard conjecture.  \cite{KK1} also gives partial results on the generalized Cartan-Hadamard conjecture in dimension $4$ and \cite{GS1} gives results on the total curvature of level sets in Riemannian manifolds which are related to the proof of the $3$-dimensional generalized Cartan-Hadamard conjecture in \cite{Kl}.  The results in \cite{HM1} give a sharp, quantitative version of the Cartan-Hadamard conjecture in dimension $2$, and \cite{RS1,Sch1} give additional proofs of the Cartan-Hadamard conjecture in dimension $3$.  

\subsection*{Acknowledgements:} I am very happy to thank Chris Croke, Joe Fu and Rob Kusner for many helpful conversations about this work and Peter McGrath for his feedback about an early draft of this paper.


\section{The Space of Geodesics and the Secant Mapping}
\label{secants}

In this section, we will define the space of geodesics $\Gcal$ in a Cartan-Hadamard manifold $\Hcal$, the secant space $S(M)$ of a closed, oriented manifold $M$ and the secant mapping $l:S(M) \rightarrow \Gcal$ which is induced by an immersion $f:M \rightarrow \Hcal$.  We will establish several results that form the basis of the proofs of our main theorems in Sections \ref{bp_in_ch} and \ref{isoperimetric_inequalities}. \\    

We will explain how in any Cartan-Hadamard manifold $\Hcal$, the family of oriented geodesics is canonically a symplectic manifold, of dimension $2n-2$ if $\Hcal$ is $n$-dimensional.  We begin by defining the space of geodesics as the quotient of the unit tangent bundle $U(\Hcal)$ of $\Hcal$ by the geodesic flow: 

\begin{definition}
\label{space_of_G}
Let $\Hcal$ be a Cartan-Hadamard manifold, $U(\Hcal)$ its unit tangent bundle and $\zeta^{t}$ the $\R$-action on $U(\Hcal)$ given by the geodesic flow; that is, letting $\vec{u}$ be a unit tangent vector to $\Hcal$ and $\gamma_{\vec{u}}$ the geodesic in $\Hcal$ with $\gamma_{\vec{u}}'(0) = \vec{u}$, $\zeta^{t}(\vec{u}) = \gamma_{\vec{u}}'(t)$. \\  

$\Gcal = U(\Hcal) / \zeta$ is the space of oriented geodesics of $\Hcal$.  We let $\Gamma: U(\Hcal) \rightarrow \Gcal$ denote the quotient mapping.  \end{definition}  

The following gives a natural homeomorphism between $\Gcal$ and the tangent bundle to the unit sphere $U_{x}\Hcal$ in any tangent space $T_{x}\Hcal$ -- in particular, this result implies that $\Gcal$ is canonically a $(2n-2)$-dimensional smooth manifold, diffeomorphic to the tangent bundle of the $(n-1)$-sphere: 

\begin{prop}
\label{G_can_diffeom}
Given an oriented geodesic $\gamma$ in a Cartan-Hadamard manifold $\Hcal$ and $p \in \Hcal$, let $b(\gamma,p)$ be the (necessarily unique) point on $\gamma$ nearest to $p$, let $\vec{v}(\gamma,p)$ be $Exp_{p}^{-1}(b(\gamma,p))$, let $\vec{w}(\gamma,p)$ be the pre-image of $\gamma'$ at $b(\gamma,p)$ via $dExp_{p}$ and the identification of $T_{p}\Hcal$ with $T_{\vec{v}(\gamma,p)}T_{p}\Hcal$, and let $\vec{u}(\gamma,p)$ be the unit vector in $T_{p}\Hcal$ parallel to $\vec{w}$.  Let $T(U_{p}\Hcal)$ be the tangent bundle to the unit sphere in $T_{p}\Hcal$. \\ 

Then by the natural identification of $T(U_{p}\Hcal)$ with $\lbrace (\vec{u}, \vec{v}) \in U_{p}\Hcal \times T_{p}\Hcal : \vec{v} \in \vec{u}^{\perp} \rbrace$, $\vec{v}(\gamma,p)$ belongs to $T_{\vec{u}(\gamma,p)}(U_{p}\Hcal)$, and the mapping $\Gcal \rightarrow T(U_{p}\Hcal)$ sending $\gamma$ to $\left( \vec{u}(\gamma,p), \vec{v}(\gamma,p) \right)$ is a homeomorphism. \end{prop} 

Note that given an orientation of $\Hcal$, we can also define a diffeomorphism from $\Gcal$ to the total space of the tautological bundle over the Grassmannian of oriented hyperplanes in $T_{p}\Hcal$.  This gives an alternate formulation of Proposition \ref{G_can_diffeom} which will be helpful in some of our results, and we will assume throughout that $\Hcal$ is oriented.  \\  

The unit tangent bundle $U(\Hcal)$ carries a canonical contact form $\alpha$, whose exterior derivative gives a symplectic form on its kernel $\alpha^{\perp}$.  $\alpha$ and $d\alpha$ are invariant under the geodesic flow of $\Hcal$, and for $\vec{u} \in U(\Hcal)$, $d\Gamma : \alpha_{\vec{u}}^{\perp} \rightarrow T_{\Gamma(\vec{u})}\Gcal$ is a linear isomorphism, so $d\alpha$ descends to a canonically-defined symplectic form on $\Gcal$.  The top power of this symplectic form gives a canonical measure on $\Gcal$.  We adopt the notation of \cite{Pohl,BP} for this symplectic form and record these structures in the following: 

\begin{definition}
\label{G_form}
\label{G_measure} 

Let $dI$ be the canonical symplectic form on the space of geodesics $\Gcal$ in a Cartan-Hadamard manifold $\Hcal$.  We define the measure $dVol_{\Gcal}$ on $\Gcal$ to be that given by the following top-dimensional differential form:   

\begin{equation}
\label{G_measure_equation}
\displaystyle dVol_{\Gcal} = \text{\small $\frac{(-1)^{\binom{n}{2}}}{(n-1)!}$} \displaystyle dI^{(n-1)}. \medskip
\end{equation}
\end{definition}

The sign and constant in (\ref{G_measure_equation}) are chosen so that $\alpha \wedge \Gamma^*(dVol_{\Gcal}) = dVol_{\Hcal} \wedge dU$, where $dU$ is the fibrewise volume form on $U(\Hcal)$.  Note that $\Gamma^{*}(dI) = d\alpha$, and that for $n \equiv 0,1 \mod \ 4$ we have $(-1)^{\binom{n}{2}} = 1$ and for $n \equiv 2,3 \mod \ 4$ we have $(-1)^{\binom{n}{2}} = -1$. \\ 

Letting $F(\Hcal)$ denote the frame bundle of $\Hcal$, $U(\Hcal)$ is the base of a fibration from $F(\Hcal)$, which sends a frame $\lbrace E_{1}, E_{2}, \dots, E_{n} \rbrace$ to its first vector $E_{1}$.  We denote this $\Ecal: F(\Hcal) \rightarrow U(\Hcal)$.  Letting $\omega_{1}, \omega_{2}, \cdots, \omega_{n}$ and $\omega_{1}^{2}, \omega_{1}^{3}, \cdots, \omega_{n-1}^{n}$ be the canonically-defined dual and connection $1$-forms on $F(\Hcal)$, we have:  

\begin{equation*}
\displaystyle \Ecal^{*}(\alpha) = \omega_{1}, 
\end{equation*}
\begin{equation}
\label{cartan_eqn}
\displaystyle \Ecal^{*}(d\alpha) = d\omega_{1} = \sum\limits_{j=2}^{n} \omega_{1}^{j} \wedge \omega_{j}. \bigskip 
\end{equation}

In terms of $\omega_{1}, \omega_{2}, \cdots, \omega_{n}$ and $\omega_{1}^{2}, \omega_{1}^{3}, \cdots, \omega_{n-1}^{n}$, Definition \ref{G_measure} says:  

\begin{equation}
\label{Gvol_formula}	
\displaystyle \left( \Gamma \circ \Ecal \right)^{*} (dVol_{\Gcal}) = \omega_{2} \wedge \cdots \wedge \omega_{n} \wedge \omega_{1}^{2} \wedge \cdots \wedge \omega_{1}^{n}. \bigskip 
\end{equation}

The following formula for the measure on $\Gcal$ will be important in several of our results:   

\begin{prop}[Santal\'o's formula, see also \cite{Cr2,Cr,Sa,Be1}]
\label{crofton_santalo}
Let $f : M^{n-1} \rightarrow \Hcal^{n}$ be an immersion of an oriented hypersurface in a Cartan-Hadamard manifold.  Let $\Ucal(f)$ be the pull-back of the unit tangent bundle $U(\Hcal)$ of $\Hcal$ via $f$ and $\Upsilon : \Ucal(f) \rightarrow \Gcal$ the map which sends $\vec{u} \in U_{x}\Hcal$ to the geodesic $\gamma_{\vec{u}}$ in $\Hcal$ determined by $\vec{u}$.  Let $dVol_{\Ucal} = dVol_{M} \wedge dU$ be the volume form on $\Ucal(f)$. \\ 

Then $\Upsilon^{*}(dVol_{\Gcal})_{\vec{u}} = \langle dVol_{M_{x}}, dU_{\vec{u}} \rangle dVol_{\Ucal}$. \end{prop}

\begin{proof} This is an immediate consequence of the formula (\ref{Gvol_formula}) for $dVol_{\Gcal}$. \end{proof}

Note that, letting $\nu$ be a unit normal vector field to $M$ chosen so that $\nu^{\flat} \wedge dVol_{M} = dVol_{\Hcal}$, Proposition \ref{crofton_santalo} can also be written $\Upsilon^{*}(dVol_{\Gcal})_{\vec{u}} = \langle \vec{u}, \nu \rangle dVol_{\Ucal}$.  For oriented submanifolds $M_{1}^{k}, M_{2}^{n-k}$ of complementary dimensions in a Cartan-Hadamard manifold $\Hcal^{n}$ and a transverse intersection point $x \in M_{1} \cap M_{2}$, we will write $i_{x}(M_{1},M_{2})$ for the oriented intersection number of $M_{1}$ and $M_{2}$ at $x$: $i_{x}(M_{1},M_{2}) = +1$ if $ \langle dVol_{\Hcal} , dVol_{M_{1}} \wedge dVol_{M_{2}} \rangle > 0$ and $i_{x}(M_{1},M_{2}) = -1$ if $\langle dVol_{\Hcal} , dVol_{M_{1}} \wedge dVol_{M_{2}} \rangle < 0$.  For $\vec{u} \in U_{x}\Hcal$, the sign of $det(d\Upsilon)_{\vec{u}}$ is therefore equal to $i_{x}(\gamma_{\vec{u}},M)$.  In terms of the positive measures induced by the volume form on $\Gcal$ and the top-dimensional form $\langle \vec{u}, \nu \rangle dVol_{\Ucal}$ on $\Ucal(f)$, Proposition \ref{crofton_santalo} says that $\Upsilon$ is a measure-preserving mapping from $\Ucal(f)$ to $\Gcal$. \\ 

For an immersion $f : M^{n-1} \rightarrow \Hcal^{n}$ of an oriented hypersurface as in Proposition \ref{crofton_santalo}, we define the invariant $\Acal(f)$ in Theorem \ref{generalization} as follows: 

\begin{definition}
\label{average}
Let $f : M^{n-1} \rightarrow \Hcal^{n}$ be as in Proposition \ref{crofton_santalo}, let $\nu$ be the unit normal vector field to $M$ with $\nu^{\flat} \wedge dVol_{M} = dVol_{\Hcal}$ and for $\vec{u} \in U_{x}\Hcal$, let $\rho_{\vec{u}}$ be the geodesic ray in $\Hcal$ determined by $\vec{u}$.  For $x\in M$, let $\mu(x)$ be: 

\begin{equation*}
\label{mass_function_eqn}
\displaystyle \mu(x) = c_{n} \int\limits_{U_{x}\Hcal} |\langle \vec{u}, \nu \rangle|^{\frac{n}{n-2}} \chi(\lbrace y \in M : f(y) \in \rho_{\vec{u}}, \ i_{x}(\gamma_{\vec{u}},M) \neq i_{y}(\gamma_{\vec{u}},M) \rbrace) d\vec{u}, \medskip 
\end{equation*}
where $k_{n}$ is: 
\begin{equation}
\label{mass_function_constant}
\displaystyle \frac{2}{\int\limits_{S^{n-1}} |\langle \vec{u}, \nu \rangle|^{\frac{n}{n-2}} d\vec{u}}. \medskip 
\end{equation}

Let $\Acal(f)$ be the average of $\mu(x)$ over $M$:  
\begin{equation*}
\displaystyle \Acal(f) = \frac{1}{Vol(M)} \int\limits_{M} \mu(x) dVol_{M}.
\end{equation*}
\end{definition}

For $n=2$, the constant $C_{2}$ in Theorem \ref{generalization} is equal to $4\pi$.  For $n \geq 3$, $C_{n}$ is defined in terms of $k_{n}$:  

\begin{equation}
\label{generalization_constant}
\displaystyle C_{n} = \Sigma_{n-1}k_{n}^{\left(\frac{n-2}{n-1}\right)}. \bigskip 
\end{equation}

The following corollary of Proposition \ref{crofton_santalo} allows us to calculate the area of a compact hypersurface $M$ in a Cartan-Hadamard manifold $\Hcal$ as an integral over $\Gcal$, exactly as the area of a hypersurface in Euclidean space may be calculated as an integral over the space of lines: 

\begin{theorem}[Crofton's Formula]
\label{Crofton}
Let $f:M^{n-1} \rightarrow \Hcal^{n}$ be an immersion of a compact hypersurface $M$ in a Cartan-Hadamard manifold $\Hcal$.  Then:   

\begin{equation}
\label{Crofton_eqn}
\displaystyle \int\limits_{\Gcal} \chi(f^{-1}(\gamma)) dVol_{\Gcal} = 2B_{n-1} Vol(M). \medskip  
\end{equation}
\end{theorem} 

\begin{proof} Because both sides of (\ref{Crofton_eqn}) are additive over finite disjoint unions (and more generally, over unions which intersect in subsets of $M$ of measure zero) it is sufficient to prove the result for $M$ embedded in $\Hcal$ and diffeomorphic to an $(n-1)$-ball.  We can therefore suppose $M$ is oriented.  By Proposition \ref{crofton_santalo} and a change-of-variables, 

\begin{equation*}
\displaystyle \int\limits_{\Gcal} \chi(f^{-1}(\gamma)) dVol_{\Gcal} = \int\limits_{\Gcal} \scriptstyle \sum\limits_{x \in f^{-1}(\gamma)} \displaystyle \left( i_{x}(\gamma,M) \right)^{2} dVol_{\Gcal}
\end{equation*}

\begin{equation*}
\displaystyle = \int\limits_{\Gcal} \scriptstyle \sum\limits_{\vec{u} \in \Upsilon^{-1}(\gamma)} \displaystyle sgn(det(d\Upsilon)_{\vec{u}})^{2} dVol_{\Gcal} = \int\limits_{\Ucal(f)} |det(d\Upsilon)_{\vec{u}}| dVol_{\Ucal}. \bigskip 
\end{equation*}

By Fubini, this is equal to:  

\begin{align*}
\displaystyle \int\limits_{M} \int\limits_{U_{x}\Hcal} |det(d\Upsilon)_{\vec{u}}| dU dVol_{M}. \bigskip 
\end{align*}

An elementary calculation shows $\int_{U_{x}\Hcal} |det(d\Upsilon)_{\vec{u}}| dU = \frac{2\Sigma_{n-2}}{n-1} = 2B_{n-1}$.  This implies (\ref{Crofton_eqn}) for $M$ orientable and embedded, and thus by additivity for all $M$ compact and smoothly immersed in $\Hcal$. \end{proof}

We will define the secant space $S(M)$ of a closed, oriented manifold $M$ following Banchoff and Pohl and the authors cited in \cite{Pohl,BP}.  We refer to \cite[Section 4]{Pohl} and \cite{BP} for more background on the secant space.   

\begin{definition}
\label{secant_space}
Let $M^{m}$ be a closed, oriented manifold of dimension $m$.  Let $M \times M$ be the product manifold of $M$ with itself, oriented such that if $[M_{1}]$ and $[M_{2}]$ are the fundamental classes of the first and second factors respectively, $[M_{1}]\times [M_{2}]$ is the orientation class of $M \times M$.  Let $\varDelta$ be the diagonal in $M \times M$. \\ 

The secant space $S(M)$ of $M$ is the manifold $M \times M \setminus \varDelta$, together with a boundary which is constructed by identifying $M \times M \setminus \varDelta$ with the complement of a small, closed tubular neighborhood of $\varDelta$ in $M \times M$ and taking the boundary to be the boundary of this tubular neighborhood. \end{definition}

Note that, given a Riemannian metric on $M$, the boundary of $S(M)$ is diffeomorphic to the unit tangent bundle $U(M)$.  In fact, for an embedding of $M$ into a Cartan-Hadamard manifold, $\partial S(M)$ is canonically identified with the unit tangent bundle of the induced metric on $M$.  We will explain this after defining the secant mapping in Definition \ref{secant_mapping}.  The secant space $S(M)$ and the product $M \times M$ coincide up to a set of measure $0$ as $2m$-dimensional domains of integration, and we will sometimes write integrals initially defined on $S(M)$ as integrals over $M \times M$, and vice-versa \\ 

An immersion of $M$ into a Cartan-Hadamard manifold $\Hcal$ induces an almost everywhere-defined mapping from the secant space of $M$ to the space of geodesics in $\Hcal$:  

\begin{definition}
\label{secant_mapping}
Let $f:M^{m} \rightarrow \Hcal^{n}$ be an immersion of a closed, oriented manifold $M$ into a Cartan-Hadamard manifold $\Hcal$.  The secant mapping $l:S(M) \rightarrow \Gcal$ sends a pair of points $(x,y) \in M \times M$ with $f(x) \neq f(y)$ to the unique oriented geodesic $l(x,y)$ in $\Hcal$ which passes through $f(x)$ and $f(y)$, oriented from $f(x)$ to $f(y)$. \\

The secant mapping extends naturally to boundary points of $S(M)$:  if $V$ is a neighborhood of $M$ on which $f$ is an embedding, tangent lines to $f(V)$ can be realized as limiting positions to secant lines through $V \times V$.  We therefore define $l$ on $\partial S(M) = U(M)$ to be the mapping which sends unit tangent vectors to $M$ to the geodesic in $\Hcal$ which they determine via $df$. \end{definition}

\begin{remark}
\label{current_remark}
It might be interesting to develop a definition of a secant space and secant mapping for more general famlies of integral currents in Cartan-Hadamard manifolds.  In the situation considered in Question \ref{linking_question}, for an immersed, codimension $2$ submanifold $M$ which does not bound an embedded, spanning hypersurface, one may be able to prove the conjectured identity for $\int_{\Gcal} \lambda^{2}(M,\gamma) dVol_{\Gcal}$ by studying the secant mapping of an integral current $T$ with $\partial T = M$.  This development would be new even in Euclidean space and may give a more general extension of the Banchoff-Pohl inequality to the space of integral currents in $\R^{n}$.  Almgren's results in \cite{Alm1} show that many families of integral currents in $\R^{n}$ satisfy generalizations of the isoperimetric inequality.  It would be interesting to know whether the Banchoff-Pohl inequality also has a generalization to a larger family of integral currents in $\R^{n}$. \end{remark}

One can also define the secant mapping $l: M_{1} \times M_{2} \rightarrow \Gcal$ for any pair of manifolds $M_{1}, M_{2}$ disjointly embedded in $\Hcal$.  We will state Propositions \ref{key_prop} and \ref{key_prop_2} below terms of such a pair $M_{1}, M_{2}$, however in the applications to our main theorems, $M_{1}, M_{2}$ will be neighborhoods of the same closed manifold $M$ immersed in a Cartan-Hadamard space.  Given a Cartan-Hadamard manifold $\Hcal$ and points $x,y \in \Hcal$, we will let $r(x,y)$ be the distance between $x$ and $y$ and $l(x,y)$ the unique geodesic from $x$ to $y$, oriented from $x$ to $y$, as in the definitions above.  We define a linear transformation from $T_{y}\Hcal$ to $T_{x}\Hcal$ in terms of the Jacobi fields along $l(x,y)$: 

\begin{definition}
\label{Jac_trans}
For $\vec{v} \in T_{y}\Hcal$ orthogonal to $l'$, let $\Jcal(\vec{v})$ be the unique vector in $T_{x}\Hcal$ such that the Jacobi field $J(t)$ along $l$ with $J(0) = 0, J'(0) = \Jcal(\vec{v})$ satisfies $J(r(x,y)) = \vec{v}$. \end{definition}

We will often extend the definition of $\Jcal$ to all of $T_{y}\Hcal$ by setting $\Jcal(l') = 0$. \\ 

For $\mathcal{K} \leq 0$, we will write $\sn_{\Kcal}(r)$ for the length of a Jacobi field with initial conditions $J(0) = 0$, $|J'(0)| = 1$ in a model space with constant curvature $\mathcal{K}$.  Thus, if $\mathcal{K} < 0$, then $\sn_{\Kcal}(r) = \scriptstyle \frac{1}{\sqrt{|\Kcal|}} \sinh(\sqrt{|\Kcal|} r)$ as in Theorems \ref{main_thm_1}.B and \ref{main_thm_2}.B, and if  $\mathcal{K} = 0$ then  $\sn_{\Kcal}(r) = r$.  We will write $\cs_{\Kcal}(r)$ for $\sn_{\Kcal}'(r)$ and $\ct_{\Kcal}(r)$ for $\frac{\cs_{\Kcal}(r)}{\sn_{\Kcal}(r)}$.  At many points, however, it will be convenient to distinguish the case $\mathcal{K} = 0$ from the case $\mathcal{K} < 0$.  In these situations, we will write $r$ rather than $\sn_{\Kcal}(r)$ when $\mathcal{K} = 0$.  In the proof of Proposition \ref{Jac_trans_bd} and our main results, we will use several facts from the proof of the Rauch comparison theorem based on the index lemma, which is presented in \cite[Chapter 10]{doC1}.  However, our convention for the curvature tensor is the opposite of that in \cite{doC1}:  we will write $R(X,Y)Z$ for $\nabla_{X}\nabla_{Y}Z - \nabla_{Y}\nabla_{X}Z - \nabla_{[X,Y]}Z$ and $R(X,Y,Z,W)$ for $\langle R(X,Y)Z,W \rangle$.  

\begin{prop}
\label{Jac_trans_bd}
Let $x,y$ be points in a Cartan-Hadamard manifold $\Hcal^{n}$ as above, and suppose the sectional curvature of $\Hcal$ is bounded above by $\mathcal{K} \leq 0$. \\ 

Then $\Jcal$ is an orientation-preserving linear isomorphism from $l^{\perp} \subseteq T_{y}\Hcal$ to $l^{\perp} \subseteq T_{x}\Hcal$ and $\det(\Jcal) \leq \frac{1}{\sn^{n-1}_{\Kcal}(r)}$.  More generally, for any $m$-dimensional subspace $V$ of $T_{y}\Hcal$ orthogonal to $l'$, $\Jcal$ is an orientation-preserving linear isomorphism from $V$ to its image, and $\det(\Jcal|_{V}) \leq \frac{1}{\sn^{m}_{\Kcal}(r)}$.  The equality $\det(\Jcal|_{V}) = \frac{1}{\sn^{m}_{\Kcal}(r)}$ holds if and only if the following conditions are satisfied: 
\medskip
\begin{flushleft}
There is an $m$-dimensional space of Jacobi fields $J$ along $l(x,y)$ with $J(0) = 0, J(r) \in V$, all of which realize equality in the Rauch comparison theorem with curvature bounded above by $\mathcal{K}$.  For Jacobi fields $J_1, J_2$ belonging to this subspace, we have:  
\end{flushleft}
\begin{equation}
\displaystyle J_{1}'(0), J_{2}'(0) \ \text{are orthogonal in} \ T_{x}\Hcal \iff \langle J_{1}(t), J_{2}(t) \rangle \equiv 0 \ \text{along} \ l(x,y). \medskip  
\end{equation}

Moreover, if equality holds for $\Jcal: V \rightarrow \Jcal(V)$ then, letting $\widetilde{\Jcal}$ be the mapping from $l^{\perp} \subseteq T_{x}\Hcal$ to $l^{\perp} \subseteq T_{y}\Hcal$ in Definition \ref{Jac_trans}, $\widetilde{\Jcal}(\Jcal(V)) = V$, equality holds for $\widetilde{\Jcal}:\Jcal(V) \rightarrow V$ and $\widetilde{\Jcal} \circ \Jcal = \frac{1}{\sn_{\Kcal}^{2}(r)} Id_{V}$.  In addition, if $m \leq (n-2)$, there is an $(n-m-1)$-dimensional family of Jacobi fields $J$ along $l(x,y)$ with $J(0) = 0 \in T_{x}\Hcal, J'(0) \in \Jcal(V)^{\perp}$ and $J(r) \in V^{\perp}$. \end{prop}

\begin{proof} The characterization of $\Jcal$ as an orientation-preserving linear isomorphism and the inequality $\det(\Jcal) \leq \frac{1}{\sn^{m}_{\Kcal}(r)}$ follow immediately from the basic properties of Jacobi fields and the Rauch comparison theorem. \\  

Suppose $\det(\Jcal|_{V}) = \frac{1}{\sn^{m}_{\Kcal}(r)}$. \\ 

Then for any $\vec{v} \in V$, letting $J_{\vec{v}}$ be the Jacobi field along $l(x,y)$ with $J_{\vec{v}}(0) = 0$ and $J_{\vec{v}}(r) = \vec{v}$, $J_{\vec{v}}$ must realize equality in the Rauch comparison theorem with curvature bounded above by $\mathcal{K}$.  Moreover, for $\vec{v},\vec{w} \in V$, $\Jcal(\vec{v}), \Jcal(\vec{w})$ are orthogonal in $\Jcal(V)$ precisely if $\vec{v}, \vec{w}$ are orthogonal in $V$.  Letting $e_{1}, e_{2}, \dots, e_{m}$ be an orthonormal basis for $V$, $\Jcal(e_{1}), \Jcal(e_{2}), \dots, \Jcal(e_{m})$ is therefore an orthogonal basis for $\Jcal(V)$, with $||\Jcal(e_{i})|| = \frac{1}{\sn_{\Kcal}(r)}$ for $i = 1, 2, \dots, m$. \\  

Let $J_{1}, J_{2}, \dots, J_{m}$ be the Jacobi fields along $l(x,y)$ with $J_{i}(0) = 0, J_{i}(r) = e_{i}$, in other words, with $\Jcal(e_{i}) = J_{i}'(0)$.  For $t_{0} \in (0,r]$, let $\mathscr{I}_{t_{0}}$ be the vector space of all Jacobi fields along $l|_{[0,t_{0}]}$ with $J(0) = 0, \langle J(t), l'(t) \rangle \equiv 0$ for $t \in (0,t_{0}]$, and let $I_{t_{0}}$ be the index form of $l$, defined as follows for any piecewise-smooth vector fields $V,W$ along $l$:

\begin{align}
\displaystyle  I_{t_{0}}(V,W) = \int\limits_{0}^{t_{0}} \langle V',W' \rangle - R(V,l',l',W) dt. \bigskip  
\end{align}

Then $I_{t_{0}}$ is a positive-definite, symmetric bilinear form on $\mathscr{I}_{t_0}$, and for $J_{a}, J_{b} \in \mathscr{I}_{t_{0}}$, $I_{t_{0}}(J_{a},J_{b}) = \langle J_{a}(t_{0}), J_{b}'(t_{0}) \rangle$.  Defining a positive-definite inner product $\langle \langle \ \cdot \ , \cdot \ \rangle \rangle$ on $\mathscr{I}_{t_{0}}$ by setting $\langle \langle J_{a}, J_{b} \rangle \rangle = \langle J_{a}'(0), J_{b}'(0) \rangle$, we can diagonalize $I_{t_{0}}$ relative to $\langle \langle \ \cdot \ , \cdot \ \rangle \rangle$.  It follows from the proof of the Rauch comparison theorem via the index lemma that $J_{1}|_{[0,t_0]}, J_{2}|_{[0,t_0]}, \dots, J_{m}|_{[0,t_0]}$ are eigenvectors of $I_{t_0}$, associated to its smallest non-zero eigenvalue $\sn_{\Kcal}(t_0)\cs_{\Kcal}(t_0)$.  We can therefore complete the set of Jacobi fields $J_{1}, J_{2}, \dots, J_{m}$ to a basis $J_{1}, J_{2}, \dots, J_{n-1}$ of $\mathscr{I}_{t_0}$ composed of eigenvectors for $I_{t_{0}}$, which is orthogonal relative to the inner product $\langle \langle \ \cdot \ , \cdot \ \rangle \rangle$ above.  For $i \in \lbrace 1, 2, \dots, m \rbrace$, $j \in \lbrace 1, 2, \dots, n-1 \rbrace$, we therefore have:  

\begin{equation}
\label{jac_trans_bd_pf_eqn_1}
\displaystyle \langle J_{i}'(t_{0}), J_{j}(t_0) \rangle = I_{t_0}(J_{i}, J_{j}) = \delta_{i,j}\frac{\sn_{\Kcal}(t_0)\cs_{\Kcal}(t_0)}{\sn_{\Kcal}(r)^{2}}. \bigskip 
\end{equation}

For $J_{a}, J_{b} \in \mathscr{I}_{t_0}$, $[\frac{d}{dt}]\langle J_{a}(t), J_{b}(t) \rangle|_{t_{0}} = \langle J_{a}'(t_{0}), J_{b}(t_{0}) \rangle + \langle J_{a}(t_{0}), J_{b}'(t_{0}) \rangle = 2I_{t_0}(J_{a},J_{b})$.  Therefore, (\ref{jac_trans_bd_pf_eqn_1}) implies that $J_{1}, J_{2}, \dots, J_{m}$ are pointwise orthogonal along $l(x,y)$, and that for $t_{0} \in (0,r_{0}]$,    

\begin{equation}
\label{jac_trans_bd_pf_eqn_2}
\displaystyle J_{i}'(t_0) = \ct_{\Kcal}(t_{0}) J_{i}(t_0). \bigskip
\end{equation}

This gives a space of Jacobi fields with the properties described in the equality case.  It is straightforward to check the converse, that this implies $\det(\Jcal|_{V}) = \frac{1}{\sn^{m}_{\Kcal}(r)}$. \\  

If equality holds for $\Jcal: V \rightarrow \Jcal(V)$, let $\vec{v} \in T_{y}\Hcal$ be orthogonal to $l'$ and let $J(t)$ be the Jacobi field along $l(x,y)$ with $J(0) = 0 \in T_{x}\Hcal$, $J(r) = \vec{v}$, that is, such that $\Jcal(\vec{v}) = J'(0)$.  $J$ realizes equality in the Rauch comparison theorem, so $|J(t)| = \frac{\sn_{\Kcal}(t)}{\sn_{\Kcal}(r)}|\vec{v}|$ for $t \in [0,r]$, and $J'(t) = \ct_{\Kcal}(t) J(t)$ by (\ref{jac_trans_bd_pf_eqn_2}).  It follows by a direct calculation that $J'(t)$ is also a Jacobi field along $l(x,y)$.  Letting $\widetilde{J}(t) = (\ct_{\Kcal}(t) - \ct_{\Kcal}(r))J(t)$, $\widetilde{J}(r) = 0$ and $\widetilde{J}'(r) = \frac{-1}{\sn_{\Kcal}^{2}(r)} \vec{v}$.  Parametrizing $\widetilde{J}$ along $l$ with the opposite orientation shows that $\widetilde{\Jcal}(\Jcal(\vec{v})) = \frac{1}{\sn_{\Kcal}^{2}(r)}\vec{v}$. \\ 

If $m \leq n-2$, let $J_{1}, J_{2}, \dots, J_{n-1}$ be an orthogonal basis of eigenvectors which diagonalizes the index form $I_{r}$ with $J_{1}(r), J_{2}(r), \dots J_{m}(r)$ an orthonormal basis for $V$ as above.  It follows from (\ref{jac_trans_bd_pf_eqn_1}) and (\ref{jac_trans_bd_pf_eqn_2}) that for $i \in \lbrace 1, 2, \dots, m \rbrace$, $j \in \lbrace 1, 2, \dots, n-1 \rbrace$, we have $\langle J_{i}(r), J_{j}(r) \rangle = 0$. \end{proof}

It will be helpful to note that Proposition \ref{Jac_trans_bd} has a straightforward extension to the case in which we take $\Jcal$ to be defined on all of $T_{y}\Hcal$ with $\Jcal(l') = 0$ as above.  In this form, it provides the basis for Propositions \ref{key_prop} and \ref{key_prop_2}, which we will use in the proof of Theorems \ref{main_thm_1} and \ref{generalization}.  Propositions \ref{key_prop}, \ref{key_prop_2} and their applications to the proofs of Theorems \ref{main_thm_1} and \ref{main_thm_2} are slightly different for curves, i.e. $\dim(M) = 1$, than for higher-dimensional submanifolds.  For the case in which $\dim(M_{1}) = \dim(M_{2})=1$ in Propositions \ref{key_prop} and \ref{key_prop_2}, and in which $\dim(M) = 1$ in the proofs of Theorems \ref{main_thm_1} and \ref{main_thm_2}, we will write $dM$ instead of $dVol_{M_{1}}$ and $\dot{M}$ for the oriented unit tangent vector to $M$.  

\begin{prop}
\label{key_prop}
Let $\Hcal$ be a Cartan-Hadamard manifold with sectional curvature bounded above by $\mathcal{K} \leq 0$.  Let $M_{1}, M_{2}$ be oriented submanifolds of $\Hcal$, of dimension $m$ and disjoint from one another, and $l: M_{1} \times M_{2} \rightarrow \Gcal$ the secant mapping.  For $(x_{0},y_{0}) \in M_{1} \times M_{2}$, let $J^{\star}: T_{y_{0}}M_{2} \rightarrow T_{x_{0}}M_{1}$ be the linear transformation given by $\Jcal$ from Definition \ref{Jac_trans}, with $\Jcal(l') = 0$, followed by orthogonal projection onto $T_{x_{0}}M_{1}$.  Let $dI$ be the symplectic form on $\Gcal$. \\  

Then if $m \geq 2$, 

\begin{equation}
\label{key_prop_eqn_1}
\displaystyle l^{*}(dI)^{m} = m! (-1)^{\binom{m+1}{2}} \ \det (J^{\star}) \ dVol_{M_{1}} \wedge dVol_{M_{2}}. \medskip 
\end{equation}

If $m=1$, 

\begin{equation}
\label{key_prop_eqn_1_dim_1}
\displaystyle l^*(dI) = \pm \det (J^{\star}) dM_{1} \wedge dM_{2}, \medskip 
\end{equation}
where the sign in (\ref{key_prop_eqn_1_dim_1}) depends on the angles that the oriented tangents to $M_{1}$ and $M_{2}$ make with the oriented secant line, in a manner which will be described in the proof, cf. also Proposition \ref{key_prop_2}. \\  

In particular, letting $\sigma_{1} \in [0,\frac{\pi}{2}]$ be the angle between $l$ and $T_{x_{0}}M_{1}$ and $\sigma_{2} \in [0,\frac{\pi}{2}]$ the corresponding angle between $l$ and $T_{y_{0}}M_{2}$, for all $m$,  

\begin{equation}
\label{key_prop_eqn_2}
\displaystyle | \langle l^{*}(dI)^{m}, \ dVol_{M_{1}} \wedge dVol_{M_{2}} \rangle | \leq m! \ \sin(\sigma_{1}) \sin(\sigma_{2}) \sn^{-m}_{\Kcal}(r). \medskip  
\end{equation}

Equality holds in (\ref{key_prop_eqn_2}) if and only if the following conditions are satisfied:  letting $T_{x_{0}}M_{1}^{\star}, T_{y_{0}}M_{2}^{\star}$ be the orthogonal projections of $T_{x_{0}}M_{1}, T_{y_{0}}M_{2}$ onto $l^{\perp}$, $\Jcal(T_{y_{0}}M_{2}^{\star}) = T_{x_{0}}M_{1}^{\star}$, and equality holds for $\Jcal:T_{y_{0}}M_{2}^{\star} \rightarrow T_{x_{0}}M_{1}^{\star}$ in Proposition \ref{Jac_trans_bd}. \end{prop}

\begin{proof} Given $(x_{0},y_{0}) \in M_{1} \times M_{2}$ as above, let $E_{1}, E_{2}, \dots, E_{n}$ be an orthonormal frame field for $\Hcal$, defined in a neighborhood of $(x_{0},y_{0})$ in $M_{1} \times M_{2}$, with the following properties: \\

\begin{enumerate}[label=(\Alph*)]
\item 
\label{properties_i}
$E_{1}, E_{2}, \dots, E_{n}$ are based at $x$, in a neighborhood of $x_{0}$ in $M_{1}$.  
\bigskip 

\item 
\label{properties_ii}
$E_{1}$ points along $l(x,y)$. 
\bigskip 


\item 
\label{properties_iv}
$E_{1}, E_{2}, \dots, E_{m+1}$ span a subspace of $T_{x}\Hcal$ containing $T_{x}M_{1}$.  
\bigskip 

\item 
\label{properties_v}
If $m \geq 2$, then $\cos(\sigma_{1}) E_{1} + \sin(\sigma_{1})E_{2}, E_{3}, \dots, E_{m+1}$ is an oriented orthonormal basis for $T_{x}M_{1}$.  
\smallskip 
\begin{flushleft}
Note that when $m=1$, $\cos(\sigma_{1}) E_{1} + \sin(\sigma_{1})E_{2} = \dot{M}_{1}$ if $\langle \dot{M}_{1}, E_{1} \rangle > 0$, and $\cos(\sigma_{1}) E_{1} + \sin(\sigma_{1})E_{2} = -\dot{M}_{1}$ if $\langle \dot{M}_{1}, E_{1} \rangle < 0$.  If $\langle \dot{M}_{1}, E_{1} \rangle = 0$, we set $E_{2} = \dot{M}_{1}$.  
\end{flushleft}
\end{enumerate}
\medskip 

We let $\widetilde{E}$ be the associated map to the frame bundle $F(\Hcal)$ of $\Hcal$ from the neighborhood of $M_{1} \times M_{2}$ on which $E_{i}$ is defined, and we let $\widetilde{\omega}_{i}, \widetilde{\omega}_{i}^{j}$ be the pullback via $\widetilde{E}$ of the canonical 1-forms $\omega_{i}, \omega_{i}^{j}$ on $F(\Hcal)$ as in (\ref{G_measure_equation}) and (\ref{cartan_eqn}).  By (\ref{cartan_eqn}), we then have: \\   
\begin{equation}
\label{form_exp}
\displaystyle l^{*}(dI) = \widetilde{\omega}_{1}^{2} \wedge \widetilde{\omega}_{2} + \widetilde{\omega}_{1}^{3} \wedge \widetilde{\omega}_{3} + \cdots + \widetilde{\omega}_{1}^{n} \wedge \widetilde{\omega}_{n}.  \bigskip
\end{equation}

By \ref{properties_iv} above, $\widetilde{\omega}_{m+2}, \cdots, \widetilde{\omega}_{n}$ vanish on $T_{x}M_{1}$, and by \ref{properties_i}, $\widetilde{\omega}_{2}, \cdots, \widetilde{\omega}_{n}$ all vanish on $T_{y}M_{2}$, so: 
 
\begin{equation*}
\displaystyle l^{*}(dI)^{m} = m! \ \widetilde{\omega}_{1}^{2} \wedge \widetilde{\omega}_{2} \wedge \widetilde{\omega}_{1}^{3} \wedge \widetilde{\omega}_{3} \wedge \cdots \wedge \widetilde{\omega}_{1}^{m+1} \wedge \widetilde{\omega}_{m+1} \bigskip
\end{equation*}

\begin{equation}
\displaystyle = m! (-1)^{\binom{m+1}{2}} \ \widetilde{\omega}_{2} \wedge \cdots \wedge \widetilde{\omega}_{m+1} \wedge \widetilde{\omega}_{1}^{2} \wedge \cdots \wedge  \widetilde{\omega}_{1}^{m+1}. \bigskip
\end{equation}

If $m \geq 2$, then by \ref{properties_v}, \\
\begin{equation}
\displaystyle \widetilde{\omega}_{2} \wedge \cdots \wedge \widetilde{\omega}_{m+1}  |\scriptstyle_{T_{x}M_{1}} \displaystyle = \sin(\sigma_{1}) dVol_{M_{1}}. \bigskip
\end{equation}

When $m=1$, $\widetilde{\omega}_{2} = \pm \sin(\sigma_{1}) dM_{1}$, with the sign determined by whether the oriented tangent to $M_{1}$ at $x$ makes an acute or obtuse angle with $E_{1}$:  if $\langle \dot{M}_{1}, E_{1} \rangle \geq 0$, then $\widetilde{\omega}_{2} = \sin(\sigma_{1}) dM_{1}$ and if $\langle \dot{M}_{1}, E_{1} \rangle < 0$, then $\widetilde{\omega}_{2} = -\sin(\sigma_{1}) dM_{1}$. \\ 

Since $\widetilde{\omega}_{2}, \widetilde{\omega}_{3} \dots, \widetilde{\omega}_{m+1}$ vanish on $T_{y}M_{2}$, in order to determine $l^{*}(dI)^{m}$ on $T_{x}M_{1} \times T_{y}M_{2}$, it is enough to determine $\widetilde{\omega}_{1}^{2} \wedge \cdots \wedge  \widetilde{\omega}_{1}^{m+1}$ on $T_{y}M_{2}$.  To do so, we choose another orthonormal frame field $B_{1}, B_{2}, \dots, B_{n}$, this time based near $y_{0}$, with properties like those of $E_{i}$ above, that is: \\ 

\begin{enumerate}[label=(\alph*)]
\item 
$B_{1}, B_{2}, \dots, B_{n}$ are based at $y$, in a neighborhood of $y_{0}$ in $M_{2}$.  
\bigskip 
	
\item 
$B_{1}$ points along $l(x,y)$. 
\bigskip 
	
	
\item 
$B_{1}, B_{2}, \dots, B_{m+1}$ span a subspace of $T_{y}\Hcal$ containing $T_{y}M_{2}$.  
\bigskip 

\item 
If $m \geq 2$, then $\cos(\sigma_{2}) B_{1} + \sin(\sigma_{2})B_{2}, B_{3}, \dots, B_{m+1}$ is an oriented orthonormal basis for $T_{y}M_{2}$.  
\smallskip 
\begin{flushleft}
As in the definition of $E_{i}$ above, when $m=1$, $\cos(\sigma_{1}) B_{1} + \sin(\sigma_{1})B_{2} = \dot{M}_{2}$ if $\langle \dot{M}_{2}, B_{1} \rangle > 0$, and $\cos(\sigma_{1}) B_{1} + \sin(\sigma_{1})B_{2} = -\dot{M}_{2}$ if $\langle \dot{M}_{2}, B_{1} \rangle < 0$.  If $\langle \dot{M}_{2}, B_{1} \rangle = 0$, we again set $B_{2} = \dot{M}_{2}$.  
\end{flushleft}
\end{enumerate}
\medskip 

With $B_{1}, B_{2}, \dots, B_{n}$ defined as above, we then have: 

\begin{equation*}
\displaystyle \widetilde{\omega}_{1}^{2}\left(\cos(\sigma_{2})B_{1} + \sin(\sigma_{2})B_{2}\right) = \sin(\sigma_{2}) \langle E_{2}, \Jcal(B_{2}) \rangle. \bigskip 
\end{equation*}

When $m \geq 2$, then for $j,k \in \lbrace 3, \dots, m+1 \rbrace$, we have $\widetilde{\omega}_{1}^{j}(B_{k}) = \langle E_{j}, \Jcal(B_{k}) \rangle$.  Similarly, $\widetilde{\omega}_{1}^{2}(B_{k}) = \langle E_{2}, \Jcal(B_{k}) \rangle$ and $\widetilde{\omega}_{1}^{j}(\cos(\sigma_{2})B_{1} + \sin(\sigma_{2})B_{2}) = \langle E_{j}, \Jcal(\cos(\sigma_{2})B_{1} + \sin(\sigma_{2})B_{2}) \rangle = \sin(\sigma_{2}) \langle E_{j}, \Jcal(B_{2}) \rangle$.  Therefore, when $m \geq 2$,  

\begin{equation*}
\displaystyle \widetilde{\omega}_{1}^{2} \wedge \cdots \wedge \widetilde{\omega}_{1}^{m+1} (\cos(\sigma_{2})B_{1} + \sin(\sigma_{2})B_{2}, \dots, B_{m+1}) = \sin(\sigma_{2}) \det(\langle E_{j}, \Jcal(B_{k}) \rangle )_{j,k=2,\dots, m+1}. \bigskip
\end{equation*}

this implies: 

\begin{equation*}
\displaystyle l^{*}(dI)^{m} = m! (-1)^{\binom{m+1}{2}} \sin(\sigma_{1}) \sin(\sigma_{2}) \det(\langle E_{j}, \Jcal(B_{k}) \rangle ) dVol_{M_{1}} \wedge dVol_{M_{2}} \end{equation*}

\begin{equation}
\label{key_prop_pf_eqn_3} 
\displaystyle  = m! (-1)^{\binom{m+1}{2}} \det(J^{\star}) dVol_{M_{1}} \wedge dVol_{M_{2}}. \bigskip
\end{equation}

This is (\ref{key_prop_eqn_1}). \\ 

When $m=1$, we have: 

\begin{equation}
\label{key_prop_pf_eqn_1}  
\displaystyle l^*(dI) = \pm \sin(\sigma_1) \sin(\sigma_2) \langle E_{2},\Jcal (B_{2}) \rangle dM_{1} \wedge dM_{2}, \bigskip
\end{equation}
where the sign in (\ref{key_prop_pf_eqn_1}) is $+$ if $\langle \dot{M}_{1}, E_{1} \rangle \langle \dot{M}_{2}, B_{1} \rangle > 0$ and $-$ if $\langle \dot{M}_{1}, E_{1} \rangle \langle \dot{M}_{2}, B_{1} \rangle < 0$.  If $\langle \dot{M}_{1}, E_{1} \rangle \langle \dot{M}_{2}, B_{1} \rangle = 0$, the sign may be positive or negative, but because this only occurs along a subset of $M_{1} \times M_{2}$ of measure zero, we do not need to determine a convention for this case.  This is (\ref{key_prop_eqn_1_dim_1}). \\ 

By Proposition \ref{Jac_trans_bd}, (\ref{key_prop_pf_eqn_1}) and (\ref{key_prop_pf_eqn_3}), $|\det(J^{\star})| \leq \sin(\sigma_{1}) \sin(\sigma_{2}) \sn^{-m}_{\Kcal}(r)$, which establishes (\ref{key_prop_eqn_2}).  Equality in (\ref{key_prop_eqn_2}) requires that $|\det(J^{\star})| = \sin(\sigma_{1}) \sin(\sigma_{2}) \sn^{-m}_{\Kcal}(r)$, which is equivalent to having $\Jcal(T_{y_{0}}M_{2}^{\star}) = T_{x_{0}}M_{1}^{\star}$ with equality in Proposition \ref{Jac_trans_bd}. \end{proof} 

The following consequence of the proof of Proposition \ref{key_prop} will be used in the proof of Theorems \ref{main_thm_1} and \ref{generalization}:  

\begin{corollary}
\label{sign_cor}
Let $M_1, M_2$ and $\Hcal$ be as in Proposition \ref{key_prop}, with $\dim(\Hcal) = n \geq 3$ and $\dim(M_1) = \dim(M_2) = n-1$, and let $(x,y) \in M_1 \times M_2$ such that $l(x,y)$ meets $M_{1}$ and $M_{2}$ transversely at $x$ and $y$.  Then the sign of the Jacobian of $l: (M_1 \times M_2, dVol_{M_1} \wedge dVol_{M_2}) \rightarrow (\Gcal, dVol_{\Gcal})$ at $(x,y)$ is $i_{x}(l,M_1) i_{y}(l,M_2)$. \end{corollary}

The following result complements Proposition \ref{key_prop} and will also be used in the proof of Theorems \ref{main_thm_1} and \ref{main_thm_2}:  

\begin{prop}
\label{key_prop_2}
Let $\Hcal^{n}$, $M_{1}$, $M_{2}$ and $l: M_{1} \times M_{2} \rightarrow \Gcal$ and the associated vector fields and functions $E_{i}, B_{j}, \sigma_{1}, \sigma_{2}$ be as in Proposition \ref{key_prop}.  For $(x_{0},y_{0})$ in $M_{1} \times M_{2}$ such that $l(x_{0},y_{0})$ does not meet $T_{x_{0}}M_{1}$ or $T_{y_{0}}M_{2}$ orthogonally, define $\overline{J}^{\star}: T_{y_{0}}M_{2} \rightarrow T_{x_{0}}M_{1}$ as follows: 
\medskip
\begin{flushleft}
Let $\overline{T_{x_{0}}M_{1}}^{\star}, \overline{T_{y_{0}}M_{2}}^{\star}$ be the intersections of $T_{x_{0}}M_{1}, T_{y_{0}}M_{2}$ with the orthogonal complement to $l'$ in $T_{x_{0}}\Hcal, T_{y_{0}}\Hcal$. 
\end{flushleft}
\medskip
\begin{itemize}
\item On $\overline{T_{y_{0}}M_{2}}^{\star}$, $\overline{J}^{\star}$ is $\Jcal$, followed by orthogonal projection to $\overline{T_{x_{0}}M_{1}}^{\star}$.  
\medskip 
	
\item On the $1$-dimensional subspace complementary to $\overline{T_{y_{0}}M_{2}}^{\star}$ in $T_{y_{0}}M_{2}$, $\overline{J}^{\star}$ is orthogonal projection onto $span(l')$, followed by parallel translation along $l$ to $x_{0}$, followed by orthogonal projection onto $T_{x_{0}}M_{1}$.  
\begin{flushleft}
Note that for $m=1$, $\overline{T_{x_{0}}M_{1}}^{\star}, \overline{T_{y_{0}}M_{2}}^{\star}$ consist only of the origin and this case describes the entirety of $\overline{J}^{\star}$.  
\end{flushleft}
\end{itemize}
\medskip 

Let $\widetilde{\omega}_{1}$ be the pullback to $M_1 \times M_2$ via $E_{1}$ of the canonical $1$-form on $U(\Hcal)$, as in Proposition \ref{key_prop}. \\ 

Then for $m \geq 2$, 
 
\begin{equation}
\label{key_prop_2_formula}
\displaystyle dr \wedge \widetilde{\omega}_{1} \wedge l^{*}(dI)^{m-1} = (m-1)!(-1)^{\binom{m+1}{2}} \det(\overline{J}^{\star}) \ dVol_{M_{1}} \wedge dVol_{M_{2}}. \medskip 
\end{equation}

In particular, 

\begin{equation}
\label{key_prop_2_formula_2}
\displaystyle |\langle  dr \wedge \widetilde{\omega}_{1} \wedge l^{*}(dI)^{m-1}, \ dVol_{M_{1}} \wedge dVol_{M_{2}} \rangle| \leq (m-1)! \cos(\sigma_{1}) \cos(\sigma_{1}) \medskip  \sn^{-(m-1)}_{\Kcal}(r). 
\end{equation}

Equality holds in (\ref{key_prop_2_formula_2}) if and only if $\Jcal(\overline{T_{y_{0}}M_{2}}^{\star}) = \overline{T_{x_{0}}M_{1}}^{\star}$ and equality holds in Proposition \ref{Jac_trans_bd}. \\

For $m=1$, 

\begin{equation}
\label{key_prop_2_formula_dim_1}
\displaystyle dr \wedge \widetilde{\omega}_{1}  = \pm \det(\overline{J}^{\star}) dM_{1} \wedge dM_{2} = \pm \cos(\sigma_{1})\cos(\sigma_{2}) dM_{1} \wedge dM_{2}, \medskip 
\end{equation}
where the sign in (\ref{key_prop_2_formula_dim_1}) is $+$ if $\langle \dot{M}_{1}, E_{1} \rangle \langle \dot{M}_{2}, B_{1} \rangle > 0$ and $-$ if $\langle \dot{M}_{1}, E_{1} \rangle \langle \dot{M}_{2}, B_{1} \rangle < 0$, as in Proposition \ref{key_prop}. \end{prop}

\begin{proof} We define $\widehat{E}_{1}, \widehat{E}_{2}$ on the same neighborhood in $M_{1} \times M_{2}$ as $E_{1}, E_{2}, \dots, E_{n}$ above, as follows: 
 
\begin{equation}
\displaystyle \widehat{E}_{1} = \sin(\sigma_{1}) E_{1} - \cos(\sigma_{1}) E_{2}, \widehat{E}_{2} = \cos(\sigma_{1}) E_{1} + \sin(\sigma_{1}) E_{2}. \bigskip
\end{equation}

For $m \geq 2$, $\widehat{E}_{2}, E_{3}, \dots, E_{m+1}$ is then an oriented orthonormal basis for $T_{x}M_{1}$.  For $m=1$, $\widehat{E}_{2}$ is $\dot{M}_{1}$ if $\langle \dot{M}_{1}, E_{1} \rangle > 0$ and is $-\dot{M}_{1}$ if $\langle \dot{M}_{1}, E_{1} \rangle < 0$ and, as in the proof of Proposition \ref{key_prop}, if $\langle \dot{M}_{1}, E_{1} \rangle = 0$ then $\widehat{E}_{2} = \dot{M}_{1}$ so that $\widehat{E}_{1}, \widehat{E}_{2}$ coincide with $E_{1}, E_{2}$.  For all $m$, in all cases, $\widehat{E}_{1}$ is normal to $T_{x}M_{1}$ and $\widehat{E}_{1}, \widehat{E}_{2}$ is an orthonormal basis for the span of $E_{1}, E_{2}$, with the same orientation as $E_{1}, E_{2}$. \\ 

We likewise define $\widehat{B}_{1}, \widehat{B}_{2}$ on the same neighborhood as $B_{1}, B_{2}, \dots, B_{n}$, by the equivalent of the definition of $\widehat{E}_{1}, \widehat{E}_{2}$:

\begin{equation}
\displaystyle \widehat{B}_{1} = \sin(\sigma_{2}) B_{1} - \cos(\sigma_{2}) B_{2}, \widehat{B}_{2} = \cos(\sigma_{2}) B_{1} + \sin(\sigma_{2}) B_{2}. \bigskip
\end{equation}

In this notation, for $m \geq 2$, $\overline{T_{x}M_{1}}^{\star} = span(E_{3}, \dots, E_{m+1})$ and $\overline{T_{y}M_{2}}^{\star} = span(B_{3}, \dots, B_{m+1})$, and for all $m$, 

\begin{equation}
\label{key_prop_2_proof_eqn}
\displaystyle \overline{J}^{\star}(\widehat{B}_{2}) = \cos(\sigma_{1}) \cos(\sigma_{2}) \widehat{E}_{2}. \bigskip
\end{equation}

For $m \geq 2$, we then have $\det(\overline{J}^{\star}) = \cos(\sigma_{1}) \cos(\sigma_{2}) \det(\overline{J}^{\star}|_{\overline{T_{y}M_{2}}^{\star}})$. \\ 

We let $\widehat{\omega}_{1}, \widehat{\omega}_{2}$ be the dual $1$-forms to $\widehat{E}_{1}, \widehat{E}_{2}$ and $\widehat{\beta}_{1}, \widehat{\beta}_{2}$ the dual $1$-forms to $\widehat{B}_{1}, \widehat{B}_{2}$ and we introduce the notation $\widetilde{\beta}_{3}, \dots, \widetilde{\beta}_{m+1}$ for the dual $1$-forms to $B_{3}, \dots, B_{m+1}$.  For $m \geq 2$, we then have $dVol_{M_1} = \widehat{\omega}_{2} \wedge \widetilde{\omega}_{3} \wedge \cdots \wedge \widetilde{\omega}_{m+1}$ and $dVol_{M_2} = \widehat{\beta}_{2} \wedge \widetilde{\beta}_{3} \wedge \cdots \wedge \widetilde{\beta}_{m+1}$.  For $m=1$, $\widehat{\omega}_{2} = dM_{1}$ if $\langle \dot{M}_{1}, E_{1} \rangle \geq 0$ and $\widehat{\omega}_{2} = -dM_{1}$ if $\langle \dot{M}_{1}, E_{1} \rangle < 0$.  Likewise, $\widehat{\beta}_{2} = dM_{2}$ if $\langle \dot{M}_{2}, B_{1} \rangle \geq 0$ and $\widehat{\beta}_{2} = -dM_{2}$ if $\langle \dot{M}_{2}, B_{1} \rangle < 0$. \\ 

For $m \geq 2$, we then have $dr = \cos(\sigma_{2}) \widehat{\beta}_{2} - \cos(\sigma_{1}) \widehat{\omega}_{2}$, and therefore: 

\begin{equation}
\displaystyle dr \wedge \widetilde{\omega}_{1} \wedge dI^{m - 1} = \cos(\sigma_{1}) \cos(\sigma_{2}) \ \widehat{\beta}_{2} \wedge \widehat{\omega}_{2} \wedge \left( \widetilde{\omega}_{1}^{2} \wedge \widetilde{\omega}_{2} + \cdots + \widetilde{\omega}_{1}^{m+1} \wedge \widetilde{\omega}_{m+1} \right)^{m-1}. \bigskip
\end{equation}

Because $\widetilde{\omega}_{2} = \cos(\sigma_{1}) \widehat{\omega}_{1} - \sin(\sigma_{1}) \widehat{\omega}_{2}$, and because $\widehat{\omega}_{1}$ is zero on $T_{x}M_{1} \times T_{y}M_{2}$, we then have:  

\begin{equation*}
\displaystyle dr \wedge \widetilde{\omega}_{1} \wedge dI^{m - 1} = \cos(\sigma_{1}) \cos(\sigma_{2}) \ \widehat{\beta}_{2} \wedge \widehat{\omega}_{2} \wedge \left(\widetilde{\omega}_{1}^{3} \wedge \widetilde{\omega}_{3} + \cdots + \widetilde{\omega}_{1}^{m+1} \wedge \widetilde{\omega}_{m+1} \right)^{m-1} \medskip
\end{equation*}

\begin{equation*}
\displaystyle = (m-1)! \cos(\sigma_{1}) \cos(\sigma_{2}) \ \widehat{\beta}_{2} \wedge \widehat{\omega}_{2} \wedge \widetilde{\omega}_{1}^{3} \wedge \widetilde{\omega}_{3} \wedge \cdots \wedge \widetilde{\omega}_{1}^{m+1} \wedge \widetilde{\omega}_{m+1} 
\end{equation*}

\begin{equation}
\label{key_prop_2_eqn}
\displaystyle = (-1)^{\binom{m+1}{2}}(m-1)! \cos(\sigma_{1}) \cos(\sigma_{2}) \ dVol_{M_{1}} \wedge \widehat{\beta}_{2} \wedge \widetilde{\omega}_{1}^{3} \wedge \cdots \wedge \widetilde{\omega}_{1}^{m+1}. \bigskip
\end{equation}

By the properties of $\widetilde{\omega}_{1}^{j}$ established in the proof of Proposition \ref{key_prop}, 
 
\begin{equation*}
\displaystyle \widetilde{\omega}_{1}^{3} \wedge \cdots \wedge \widetilde{\omega}_{1}^{m+1} = \det(\overline{J}^{\star}|_{T_{y}M_{2} \cap (l')^{\perp}}) \widetilde{\beta}_{3} \wedge \cdots \wedge \widetilde{\beta}_{m+1} 
\end{equation*}

\begin{equation}
\label{key_prop_2_proof_eqn_2}
\displaystyle  = \det(\langle E_{j}, \Jcal(B_{k}) \rangle )_{j,k=3,\dots, m+1} \ \widetilde{\beta}_{3} \wedge \cdots \wedge \widetilde{\beta}_{m+1}. \bigskip
\end{equation}

This then implies that $dr \wedge \widetilde{\omega}_{1} \wedge dI^{m - 1} = (-1)^{\binom{m+1}{2}}(m-1)! \det(\overline{J}^{\star}) \ dVol_{M_{1}} \wedge dVol_{M_{2}}$, which is (\ref{key_prop_2_formula}).  The inequality (\ref{key_prop_2_formula_2}) then follows from (\ref{key_prop_2_proof_eqn_2}) and Proposition \ref{Jac_trans_bd}, with equality preciesly if $\Jcal(\overline{T_{y_{0}}M_{2}}^{\star}) = \overline{T_{x_{0}}M_{1}}^{\star}$ with equality in Proposition \ref{Jac_trans_bd}. \\ 

For $m=1$, we also have $\widetilde{\omega}_{1} = \cos(\sigma_{1}) \widehat{\omega}_{2}$, which is equal to $\cos(\sigma_{1}) dM_{1}$ if $\langle \dot{M}_{1}, E_{1} \rangle \geq 0$ and is equal to $-\cos(\sigma_{1}) dM_{1}$ if $\langle \dot{M}_{1}, E_{1} \rangle < 0$.  There are four possible expressions for $dr$ in terms of $\cos(\sigma_{2}) \widehat{\beta}_{2}$ and $\cos(\sigma_{1}) \widehat{\omega}_{2}$, depending on the signs of $\langle \dot{M}_{1}, E_{1} \rangle$ and $\langle \dot{M}_{2}, B_{1} \rangle$, however only the sign of the term involving $\cos(\sigma_{2}) \widehat{\beta}_{2}$ is relevant to $dr \wedge \widetilde{\omega}_{1}$.  Therefore, 
  
\begin{equation}
\label{key_prop_2_eqn_2}
\displaystyle dr \wedge \widetilde{\omega}_{1} = \pm \cos(\sigma_{1}) \cos(\sigma_{2}) dM_{1} \wedge dM_{2}, \bigskip
\end{equation}
where the sign in (\ref{key_prop_2_eqn_2}) is again $+$ if $\langle \dot{M}_{1}, E_{1} \rangle \langle \dot{M}_{2}, B_{1} \rangle > 0$ and $-$ if $\langle \dot{M}_{1}, E_{1} \rangle \langle \dot{M}_{2}, B_{1} \rangle < 0$.  As in Proposition \ref{key_prop}, because $\langle \dot{M}_{1}, E_{1} \rangle \langle \dot{M}_{2}, B_{1} \rangle = 0$ on a subset of $M_{1} \times M_{2}$ of measure zero, we do not need to determine the sign of (\ref{key_prop_2_eqn_2}) in this case.  By (\ref{key_prop_2_proof_eqn}), the result in (\ref{key_prop_2_eqn_2}) is (\ref{key_prop_2_formula_dim_1}). \end{proof}

\begin{remark}
\label{key_prop_2_remark}

The condition in Proposition \ref{key_prop_2}, that $l(x_{0},y_{0})$ does not meet either $T_{x_{0}}M_{1}$ or $T_{y_{0}}M_{2}$ orthogonally, holds almost everywhere on $M_{1} \times M_{2}$.  It is natural to extend the definition of $\det(\overline{J}^{\star}) \ dVol_{M_{1}} \wedge dVol_{M_{2}}$ to be $0$ at points where $l(x_{0},y_{0})$ is orthogonal to either $T_{x_{0}}M_{1}$ or $T_{y_{0}}M_{2}$.  Proposition \ref{key_prop_2} also holds in this case, and we will adopt this definition below. \end{remark}

The following result will allow us to use Propositions \ref{key_prop} and \ref{key_prop_2} in concert in the proof of our main results:  

\begin{prop}
\label{stokes_prop}

Let $M$ be a closed, oriented $m$-dimensional manifold immersed in a Cartan-Hadamard manifold $\Hcal$ and $l:S(M) \rightarrow \Gcal$ its secant mapping.  Let $f:\R_{\geq 0} \rightarrow \R_{\geq 0}$ be a smooth function with $f(r) = o(r^{m+1}), f'(r) = o(r^{m})$ as $r \rightarrow 0$, and let $\widetilde{\omega}_{1}$ be as in Propositions \ref{key_prop} and \ref{key_prop_2}.  Then:  

\begin{equation}
\label{stokes_prop_eqn}
\displaystyle \int\limits_{S(M)} f(r) \ l^{*}(dI)^{m} = -\int\limits_{S(M)} f'(r) \ dr \wedge \widetilde{\omega}_{1} \wedge l^{*}(dI)^{m-1}. \medskip  
\end{equation}
\end{prop}

\begin{proof} If $M$ is embedded, this follows immediately from Stokes' theorem, which gives a more general result for any smooth, real-valued $f(r)$ as follows:  integrating over $\partial S(M) = U(M)$,    

\begin{equation*}
\displaystyle   (-1)^{\binom{m+1}{2}} (m-1)! Vol(U(M)) f(0) = \int\limits_{U(M)} f(r) \ \widetilde{\omega}_{1} \wedge l^*(dI)^{m-1} 
\end{equation*}

\begin{equation}
\label{stokes_prop_pf_eqn}
\displaystyle = \int\limits_{S(M)} f'(r) \ dr \wedge \widetilde{\omega}_{1} \wedge l^*(dI)^{m-1} + f(r) \ l^*(dI)^{m}. \bigskip
\end{equation}

The fact that $r \equiv 0$ on $\partial S(M)$ and $f(0) = 0$ then implies (\ref{stokes_prop_eqn}). \\  

For $M$ immersed but not embedded in $\Hcal$, $l^*(dI)^{m-1}$, $\widetilde{\omega}_{1} \wedge l^*(dI)^{m-1}$ and $l^*(dI)^{m}$ are no longer globally-defined differential forms on $S(M)$ because the mapping $l:S(M) \rightarrow \Gcal$ is not defined on the whole of $S(M)$.  However, the conditions $f(r) = o(r^{m+1})$, $f'(r) = o(r^{m})$ as $r \rightarrow 0$, and the descriptions of $dr$, $\widetilde{\omega}_{1} \wedge l^*(dI)^{m-1}$ and $l^*(dI)^{m}$ in Propositions \ref{key_prop} and \ref{key_prop_2}, ensure that $f(r) \ \widetilde{\omega}_{1} \wedge l^*(dI)^{m-1}$, $f'(r) \ dr \wedge \widetilde{\omega}_{1} \wedge l^*(dI)^{m-1}$ and $f(r) \ l^*(dI)^{m}$ are $C^{1}$ differential forms defined on the entire secant space $S(M)$, so that the result again follows from Stokes' theorem as above. \end{proof}


\section{The Banchoff-Pohl Inequality in Cartan-Hadamard Manifolds}
\label{bp_in_ch} 


In this section, we will prove Theorems \ref{main_thm_1} and \ref{main_thm_2}.  In fact, we will show that Theorem \ref{main_thm_1} follows from Theorem \ref{main_thm_2} and a formula for the invariant $\int_{\Hcal}w(M,p)^{2} dp$, which we will state and prove in Proposition \ref{hyp_main_thm} and which we will also use in the proof of Theorem \ref{generalization}.  

\begin{proof}[Proof of Theorem \ref{main_thm_2}]
	
We will first show that, in the notation of Section \ref{secants}:  

\begin{equation}
\label{bpch_pf_eqn_1} 
(m+1) \displaystyle \int\limits_{S(M)} r \ l^{*}(dI)^{m} = \int\limits_{S(M)} r \ l^{*}(dI)^{m} - m \ dr \wedge \widetilde{\omega}_{1} \wedge l^{*}(dI)^{m-1}. \bigskip 
\end{equation}	

If $M$ is embedded, this follows immediately from (\ref{stokes_prop_pf_eqn}) in the proof of Proposition \ref{stokes_prop}, with $f(r) = r$.  To allow for the possibility that $M$ is immersed but not embedded, let $\phi_{m}(r) = \frac{r^{m+2}}{1 + r^{m+1}}$.  For each $N > 0$, $( \frac{1}{N} ) \phi_{m}(Nr) = o(r^{m+1})$ and $\phi_{m}'(Nr) = o(r^{m})$ as $r \rightarrow 0$.  By Proposition \ref{stokes_prop}, we therefore have: 

\begin{equation*}
(m+1) \displaystyle \int\limits_{S(M)} \scriptstyle \frac{1}{N} \displaystyle \phi_{m}(Nr) l^*(dI)^{m} =  \int\limits_{S(M)} \scriptstyle \frac{1}{N} \displaystyle \phi_{m}(Nr) l^*(dI)^{m} - m \int\limits_{S(M)} \phi_{m}'(Nr) dr \wedge \widetilde{\omega}_{1} \wedge l^*(dI)^{m-1}. \medskip   
\end{equation*}

Taking the limit as $N \rightarrow \infty$, (\ref{bpch_pf_eqn_1}) then follows from the dominated convergence theorem. \\  

By (\ref{bpch_pf_eqn_1}), we then have:  

\begin{equation*}
(-1)^{(\binom{m+1}{2} + 1)}(m+1) \displaystyle \int\limits_{S(M)} r \ l^{*}(dI)^{m}  
\end{equation*}

\begin{equation*}
= (-1) \displaystyle \int\limits_{S(M)} \left( r \langle l^*(dI)^{m}, dVol_{S(M)} \rangle - m \langle dr \wedge \widetilde{\omega}_{1} \wedge l^*(dI)^{m-1}, dVol_{S(M)} \rangle \right) dVol_{S(M)} 
\end{equation*}

\begin{equation}
\label{bpch_pf_key_prop_eqn}
\displaystyle \leq \int\limits_{S(M)} \left( r | \langle l^*(dI)^{m}, dVol_{S(M)} \rangle | + m | \langle dr \wedge \widetilde{\omega}_{1} \wedge l^*(dI)^{m-1}, dVol_{S(M)}\rangle | \right) dVol_{S(M)}. \bigskip
\end{equation}

By Propositions \ref{key_prop} and \ref{key_prop_2}, this implies:  

\begin{equation*}
\frac{(-1)^{(\binom{m+1}{2} + 1)}(m+1)}{m!} \displaystyle \int\limits_{S(M)} r \ l^{*}(dI)^{m}  
\end{equation*}

\begin{equation}
\label{bpch_pf_eqn_2}	
\displaystyle \leq \int\limits_{S(M)} \frac{1}{\sn_{\Kcal}(r)^{m-1}} \left( \frac{r}{\sn_{\Kcal}(r)} \sin(\sigma_{1}) \sin(\sigma_{2}) + \cos(\sigma_{1}) \cos(\sigma_{2}) \right) dVol_{S(M)}. \bigskip   
\end{equation}

In any Cartan-Hadamard manifold $\Hcal$, we therefore have: 

\begin{equation*}
\frac{(-1)^{(\binom{m+1}{2} + 1)}(m+1)}{m!} \displaystyle \int\limits_{S(M)} r \ l^{*}(dI)^{m} \leq \int\limits_{S(M)} \frac{1}{r^{m-1}} \cos(\sigma_{1} - \sigma_{2}) dVol_{S(M)} 
\end{equation*}

\begin{equation}
\label{bpch_pf_eqn_3}
\displaystyle \leq \int\limits_{S(M)} \frac{1}{r^{m-1}} dVol_{S(M)}, \bigskip 
\end{equation}
which is (\ref{main_thm_2_eqn}).  When the curvature of $\Hcal$ is bounded above by $\mathcal{K} < 0$, (\ref{bpch_pf_eqn_2}) implies that: 

\begin{equation*}
\frac{(-1)^{(\binom{m+1}{2} + 1)}(m+1)}{m!} \displaystyle \int\limits_{S(M)} r \ l^{*}(dI)^{m} + \int\limits_{S(M)} \frac{\sn_{\Kcal}(r) - r}{\sn_{\Kcal}(r)^{m}} \sin(\sigma_{1})\sin(\sigma_{2}) dVol_{S(M)} 
\end{equation*}

\begin{equation}
\label{bpch_pf_eqn_4}
\displaystyle \leq \int\limits_{S(M)} \frac{1}{\sn_{\Kcal}(r)^{m-1}} \cos(\sigma_{2} - \sigma_{1}) dVol_{S(M)} \leq \int\limits_{S(M)} \frac{1}{\sn_{\Kcal}(r)^{m-1}} dVol_{S(M)}. \bigskip 
\end{equation}

By Proposition \ref{key_prop}, 

\begin{equation}
\label{bpch_pf_eqn_5}
 \frac{(-1)^{(\binom{m+1}{2} + 1)}}{m!} \displaystyle \int\limits_{S(M)} \left( \sn_{\Kcal}(r) - r \right) l^{*}(dI)^{m} \leq \int\limits_{S(M)} \frac{\sn_{\Kcal}(r) - r}{\sn_{\Kcal}(r)^{m}} \sin(\sigma_{1})\sin(\sigma_{2}) dVol_{S(M)}. \bigskip 
\end{equation}

By (\ref{bpch_pf_eqn_4}) and (\ref{bpch_pf_eqn_5}), we therefore have

\begin{equation}
\frac{(-1)^{(\binom{m+1}{2} + 1)}}{m!} \displaystyle \int\limits_{S(M)} \left( \sn_{\Kcal}(r) + m r \right) l^{*}(dI)^{m} \leq \int\limits_{S(M)} \frac{1}{\sn_{\Kcal}(r)^{m-1}} dVol_{S(M)}, \bigskip 
\end{equation}
which is (\ref{main_thm_2_B_eqn}). \\ 

Suppose we have equality. \\

We begin by collecting some elementary consequences:  by (\ref{bpch_pf_key_prop_eqn}), $\langle l^*(dI)^{m}, dVol_{S(M)} \rangle \leq 0$ and $\langle dr \wedge \widetilde{\omega}_{1} \wedge l^*(dI)^{m-1}, dVol_{S(M)}\rangle \geq 0$.  By Proposition \ref{key_prop}, this implies that $sgn(det(J^{\star})) = (-1)^{(\binom{m+1}{2} + 1)}$.  By Proposition \ref{key_prop_2}, when $m \geq 2$, this implies that $sgn(det(\overline{J}^{\star})) = (-1)^{\binom{m+1}{2}}$ wherever $\overline{J}^{\star}$ is defined and when $m=1$, this implies that $\langle \dot{M}_{1}, E_{1} \rangle \langle \dot{M}_{2}, B_{1} \rangle \geq 0$.  By (\ref{bpch_pf_eqn_2}), equality also implies that equality holds in Propositions \ref{key_prop} and \ref{key_prop_2} for all $(x,y)$ in $S(M)$ at which the secant mapping is defined.  Therefore, for all such $(x,y)$, 

\begin{equation}
\label{bpch_pf_eq_eqn_1}
\displaystyle det(\Jcal : T_{y}M^{\star} \rightarrow T_{x}M^{\star}) = \frac{(-1)^{(\binom{m+1}{2} + 1)}}{\sn_{\Kcal}(r)^{m}}, 
\end{equation}
\begin{equation}
\label{bpch_pf_eq_eqn_2}
\displaystyle det(\Jcal:\overline{T_{y}M}^{\star} \rightarrow \overline{T_{x}M}^{\star}) = \frac{(-1)^{(\binom{m+1}{2})}}{\sn_{\Kcal}(r)^{m-1}}. \bigskip 
\end{equation}

This then implies that, in the notation of the proof of Proposition \ref{key_prop}, $\Jcal(B_{2}) = \frac{-1}{\sn_{\Kcal}(r)}E_{2}$ wherever $\overline{J}^{\star}$ is defined as in Proposition \ref{key_prop_2}. \\ 

Finally, equality implies that for all $(x,y)$ in $S(M)$ at which the secant mapping is defined, $\sigma_{1} = \sigma_{2}$.  For the remainder of the proof, we will denote common value $\sigma$.  Letting $\overline{d}_{x}$ be the extrinsic distance function from $f(x)$ in $\Hcal$ and $d_{x}$ its restriction to $M$, this implies that all critical points of $d_{x}$ occur where $M$ intersects the image via the exponential map of $\Hcal$ of the normal space to $M$ at $x$, which we denote $Exp_{x}^{\Hcal}(T_{x}M^{\perp})$.  Because $\Jcal(T_{y}M^{\star}) = T_{x}M^{\star}$ (which says that $\Jcal(T_{y}M) = T_{x}M$ when $l(x,y)$ meets $T_{x}M$ and $T_{y}M$ orthogonally), Proposition \ref{Jac_trans_bd} implies that $M$ and $Exp_{x}^{\Hcal}(T_{x}M^{\perp})$ intersect transversely, and thus in a discrete set.  The critical points of $d_{x}$ are therefore a discrete subset of $M$. \\ 

We will prove the characterization of equality in several steps:  first, we will show that for all $x \in M$, any critical point $y_{0}$ of $d_{x}$ with $f(y_0) \neq f(x)$ is a local maximum of $d_x$.  We will use this to show that the image $f(M)$ of $f:M \rightarrow \Hcal$ is a closed, connected submanifold, homeomorphic to $S^{m}$ and canonically oriented by $f$.  We will then show that the image of $f$ is contained in a geodesic sphere in $\Hcal$.  Finally, letting $p_{0}$ be the center of this geodesic sphere, we will show that the geodesic cone $C(p_{0},M)$ with base $f(M)$ and vertex $p_{0}$ is an embedded, totally geodesic submanifold of $\Hcal$ and is isometric to a ball in an $(m+1)$-dimensional space of constant curvature $\mathcal{K}$. \\

To show that for all $x \in M$, any critical point $y_{0}$ of $d_{x}$ with $f(y_0) \neq f(x)$ is a local maximum of $d_x$, we will completely describe $Hess(d_x)$ at such critical points, as well as the second fundamental form of $M$ at $y_0$ in the normal direction given by $l'$.  In doing so, we will assume that $M$ and $f: M \rightarrow \Hcal$ are $C^{2}$, however it will be important to note that to this point, the entire proof, including propositions \ref{Jac_trans_bd}, \ref{key_prop}, \ref{key_prop_2} and \ref{stokes_prop}, is valid for $C^{1}$ submanifolds $M$ and immersions $f:M \rightarrow \Hcal$. \\ 

To establish these facts about critical points $y_{0}$ of $d_{x}$, note that for fixed $x \in M$, the function $\sigma(x,\cdot)$ can be extended to a continuous function $\overline{\sigma}$ on $\Hcal \setminus f(x)$ with values in $[0,\frac{\pi}{2}]$, measuring the angle that the geodesic segment from $f(x)$ to a point $y$ in $\Hcal$ makes with $T_{x}M$.  This function $\overline{\sigma}$ is smooth away from $\overline{\sigma}^{-1}(\frac{\pi}{2})$ and $\overline{\sigma}^{-1}(0)$, and in a neighborhood of any point $y \in \Hcal$ with $\overline{\sigma}(y) \notin \lbrace 0, \frac{\pi}{2} \rbrace$, the vectors $E_{1}, \widehat{E}_{1}, \widehat{E}_{2}$ and $E_{2}$ as in the proofs of Propositions \ref{key_prop} and \ref{key_prop_2} are well-defined in $T_{f(x)}\Hcal$.  Letting $r = \overline{d}_{x}(y)$, we can then define the vector field $B_{2}$ as in Proposition \ref{key_prop} in a neighborhood of $y \in \Hcal$ by setting $\Jcal(B_{2}) = \frac{-1}{\sn_{\Kcal}(r)}E_{2}$.  The vector field $B_{1}$ as in Proposition \ref{key_prop} is also well-defined on $\Hcal \setminus f(x)$ and coincides with $grad(\overline{d}_{x})$.  For $y$ in $M$ with $\sigma(x,y) \neq \frac{\pi}{2}$, these definitions of $\widehat{E}_{1}, \widehat{E}_{2}, E_{2}, B_{1}$ and $B_{2}$ coincide with those in the proofs of Propositions \ref{key_prop} and \ref{key_prop_2}. \\ 

The vector field $\widehat{B}_{2}$ as in the proof of Proposition \ref{key_prop_2} is also well-defined, and $\cos(\sigma) \widehat{B}_{2}$ coincides with $grad(d_{x})$, on $M \setminus \overline{\sigma}^{-1}(\lbrace 0, \frac{\pi}{2} \rbrace)$.  We claim that the gradient of the function $\sigma(x,\cdot)$ defined on $M \setminus f^{-1}(x)$ coincides with $\frac{\sin(\sigma)}{\sn_{\Kcal}(r)}\widehat{B}_{2}$:  because $B_{2}(\overline{\sigma}) = \frac{1}{\sn_{\Kcal}(r)}$ at all $y \in M \setminus f^{-1}(x)$, this is immediate when $m=1$, and when $m \geq 2$ it suffices to show that, for any $B$ in $T_{y}M$ which is orthogonal to $\widehat{B}_{2}$, $B(\sigma) = 0$.  To see this, let $0 < \sigma_{0} = \sigma(x,y) < \frac{\pi}{2}$.  The pre-image of $\sigma_{0}$ in $\Hcal$ via $\overline{\sigma}$ is the image via $Exp^{\Hcal}$ of the cone in $T_{x}\Hcal$ whose link is: 

\begin{equation}
\lbrace \cos(\sigma_0)E + \sin(\sigma_0)\nu \ : \ E \in T_{x}M,\ \nu \in T_{x}M^{\perp},\ |E| = |\nu| = 1 \rbrace. \bigskip  
\end{equation}

Because equality holds in Propositions \ref{key_prop} and \ref{key_prop_2}, $\Jcal(T_{y}M^{\star}) = T_{x}M^{\star}$ and $\Jcal(\overline{T_{y}M}^{\star}) = \overline{T_{x}M}^{\star}$.  In this notation, the geodesic segment from $x$ to $y$ is tangent to $\cos(\sigma_0)\widehat{E}_{2} + \sin(\sigma_0)\widehat{E}_{1}$ in the link of $\overline{\sigma}^{-1}(\sigma_{0})$.  Letting $\widehat{E}_{2}, E_{3}, \dots, E_{m+1}$ be an orthonormal basis for $T_{x}M$ as in Proposition \ref{key_prop_2}, the Jacobi fields determined by $J(0) = 0, J'(0) = E_{3}, \dots, E_{m+1}$ span the orthogonal complement to $\widehat{B}_{2}$ in $T_{y}M$ and are tangent to $\overline{\sigma}^{-1}(\sigma_{0})$.  This implies that $B(\sigma) = 0$ for any $B$ orthogonal to $\widehat{B}_{2}$ in $T_{y}M$, and thus that the gradient of $\sigma(x,\cdot)$ coincides with $\frac{\sin(\sigma)}{\sn_{\Kcal}(r)}\widehat{B}_{2}$ at $y$. \\  

Letting $y_{0}$ be a critical point of $d_{x}$ and $c:[0,\varepsilon) \rightarrow M$ any smooth, unit-speed path in $M$ with $c(0) = y_0$, it follows from the equality in the Rauch comparison theorem that $[\frac{d}{ds}]\sigma(x,c(s))|_{s=0} = \frac{-1}{\sn_{\Kcal}(r)}$, and therefore 

\begin{equation}
\label{path_limiting}
\displaystyle \lim\limits_{s \rightarrow 0 +} \langle c'(s), \widehat{B}_{2} \rangle = -1. \bigskip 
\end{equation}

Calculating $\langle \nabla_{\widehat{B}_{2}}\widehat{B}_{2}, \widehat{B}_{1} \rangle$ for $y$ near $y_0$ and taking the limit as $y \rightarrow y_0$, we therefore have that the second fundamental form of $M$ at $y$ in the normal direction $B_{1} = l' = grad(\overline{d}_{x})$ is diagonal, equal to $-(\ct_{\Kcal}(r) + \frac{1}{\sn_{\Kcal}(r)})Id_{T_{y_{0}}M}$.  It also follows from equality in Proposotion \ref{key_prop} that $Hess(\overline{d}_{x})|_{T_{y_{0}}M}$ is diagonal, equal to $\ct_{\Kcal}(r)Id_{T_{y_{0}}M}$.  Therefore, $Hess(d_x)$ at $y_0$ is diagonal, equal to $\frac{-1}{\sn_{\Kcal}(r)} Id_{T_{y_{0}}M}$, and $y_{0}$ is a strict local maximum of $d_{x}$. \\ 

Next, we will show that the image of $f$ in $\Hcal$ is a closed, connected submanifold of $\Hcal$ which is homeomorphic to $S^{m}$ and oriented via $f$. \\ 

Because all critical points $y_0$ of $d_x$ with $f(y_{0}) \neq f(x)$ are local maxima of $d_x$, and because $d_{x}$ is smooth on $M \setminus f^{-1}(f(x))$, every component $M_0$ of $M$ has a point $x_0$ with $f(x_0) = f(x)$, where $d_{x}|_{M_0}$ attains its minimum.  This then implies that all components of $M$ have the same image in $\Hcal$ via $f$.  This also implies that $M \setminus f^{-1}(f(x))$ retracts onto the local maxima of $d_{x}$.  Because $f$ is an immersion, $f^{-1}(f(x))$ is a discrete subset of $M$.  If $m \geq 2$, the complement of $f^{-1}(f(x))$ in each component of $M$ is connected -- this then implies that $d_{x}$ has a unique local maximum on each component of $M$, which then implies that each component of $M$ is homeomorphic to $S^{m}$ and is embedded in $\Hcal$ by $f$. \\ 

If $m = 1$, each component of $M$ is diffeomorphic to $S^{1}$.  In this case, if at least one component of $M$ is embedded, it follows from the fact that all components of $M$ have the same image via $f$ that the image of $f$ is an embedded submanifold of $M$ homeomorphic to $S^{1}$.  This is also true in the case $m=1$ if none of the components of $M$ are embedded.  To see this, let $x_{1}, x_{2}$ be points in $M$ with the following properties:  $x_{1}, x_{2}$ belong to the same component $M_{0}$ of $M$, $f(x_{1}) = f(x_{2})$, there are no other points $x_{0}$ in the oriented segment $[x_{1}, x_{2}]$ with $f(x_{0}) = f(x_{1}) = f(x_{2})$, and the length of the oriented segment $[x_{1}, x_{2}]$ is minimal among such pairs of points.  Letting $y$ be a local maximum of $d_{x_{1}} = d_{x_{2}}$, it follows from the conditions on $J^{\star}$ derived above that $M'(x_{1}) = M'(x_{2}) = -\sn_{\Kcal}(r)\Jcal(M'(y))$.  Let $\check{M}$ be a manifold which is identical to $M$, except that in place of the component $M_{0}$, $\check{M}$ has two components $M_{1}, M_{2}$ homeomorphic to $S^{1}$, and let $\check{f}:\check{M} \rightarrow \Hcal$ be a mapping which coincides with $f$ on $\check{M} \setminus (M_{1} \cup M_{2})$, which coincides with $f|_{[x_{1},x_{2}]}$ on $M_{1}$ and which coincides with $f|_{M_{0} \setminus [x_{1},x_{2}]}$ on $M_{2}$, with $M_{1}, M_{2}$ oriented so that $\check{f}$ is orientation-preserving on intervals where it coincides with $f$.  Let $\check{x}_{1}$, $\check{x}_{2}$ denote the points in $\check{M}$ corresponding to $x_{1}$, $x_{2}$.  Because equality holds for $f:M \rightarrow \Hcal$, it also holds for $\check{f}:\check{M} \rightarrow \Hcal$, and $\check{f}|_{M_{1}}$ is an embedding.  A priori, $\check{M}$ and $\check{f}$ may only be $C^{1}$ at $\check{x}_{1}$ and $\check{x}_{2}$, but the derivation above of the second fundamental form in the normal direction $l' = B_{1} = grad(\overline{d}_{x})$ and the Hessian of $d_{x}$ at $y_{0}$ remain valid for $y_{0} = \check{x}_{1}$, $\check{x}_{2}$.  The conclusion that all critical points $y_{0}$ of $d_{x}$ with $\check{f}(y_{0}) \neq \check{f}(x)$ are local maxima therefore remains valid for all $x \in \check{M}$, including those for which $\check{x}_{1}, \check{x}_{2}$ are local maxima.  As above, we therefore have that all components of $\check{M}$ have the same image via $\check{f}$.  This implies the same for $M$ via $f$, and therefore that the image of $f:M \rightarrow \Hcal$ is embedded. \\ 

When equality holds, the image of $f$ is therefore a closed, embedded submanifold of $\Hcal$ homeomorphic to $S^{m}$.  It follows from the fact that $sgn(det(J^{\star})) = (-1)^{(\binom{m+1}{2} + 1)}$ for all $x,y$ with $f(x) \neq f(y)$ that at all pairs of points with the same image via $f$, the orientations of their tangent spaces, which coincide, are consistent.  The image of $f$ is therefore a canonically oriented, closed, embedded submanifold of $\Hcal$, and we will identify $M$ with this image for the remainder of the proof. \\ 

For each $x$ in $M$, there is a unique point $A(x)$ where $d_{x}$ attains its maximum.  When $m \geq 2$, this follows from the fact established above that for each $x \in M$, $d_{x}$ has a unique local maximum on each component of $M$.  When $m = 1$, this follows from the fact that any segment between two local maxima of $d_{x}$ would contain a local minimum.  Because the geodesic chord from $x$ to $A(x)$ in $\Hcal$ meets $M$ orthogonally at both endpoints, it follows from the first variation formula that the distance between $x$ and $A(x)$ in $\Hcal$ is constant.  We will write $r_{0}$ for this common value of $r(x,A(x))$. \\  

We will now show that all geodesic chords $l(x,A(x))$, for all $x \in M$, have the same midpoint $p_{0}$ in $\Hcal$, and therefore that $M$ is contained in the geodesic sphere $S_{p_{0}}(\frac{r_{0}}{2})$ of radius $\frac{r_{0}}{2}$ about $p_{0}$ in $\Hcal$. \\  

The unit tangent to the geodesic chord from $A(x)$ to $x$ at $x$ gives a unit normal vector field to $M$, which we will denote $\nu$.  By what we have shown above, $M$ is umbilic in the normal direction $\nu$, with second fundamental form given by:  

\begin{equation}
\label{umbilic_II}
\displaystyle II_{\nu} = -(\ct_{\Kcal}(r_{0}) + \frac{1}{\sn_{\Kcal}(r_{0})})Id_{T_{x}M}. \bigskip 
\end{equation}

The mapping $x \to A(x)$ is an order $2$, fixed point free diffeomorphism of $M$ and can be written:

\begin{equation}
\label{antipodal_formula}
\displaystyle A(x) = Exp_{x}^{\Hcal}(-r_{0}\nu). \bigskip 
\end{equation}

For a tangent vector $E$ to $M$ at $x$, $dA_{x}(E)$ can be described in terms of the Jacobi field $J^{*}$ along $l(x,A(x))$ with $J^{*}(0) = E, J^{*'}(0) = \nabla_{E}(-\nu)$:  

\begin{equation}
\displaystyle dA_{x}(E) = J^{*}(r_{0}). \bigskip 
\end{equation}
 
It follows from (\ref{umbilic_II}) that $\nabla_{E}(-\nu) = -(\ct_{\Kcal}(r_{0}) + \frac{1}{\sn_{\Kcal}(r_{0})})E + (\nabla_{E}(-\nu))^{\dagger}$, where $(\nabla_{E}(-\nu))^{\dagger}$ is orthogonal to $T_{x}M$.  We will see in a moment that $(\nabla_{E}(-\nu))^{\dagger} = 0$.  This is immediate if $m = (n-1)$.  Letting $J_{E}$ be the Jacobi field along $l(x,A(x))$ with $J_{E}(0) = 0, J_{E}'(0) = E$ and letting $J^{\dagger}$ be the Jacobi field with $J^{\dagger}(0) = 0, J^{\dagger '}(0) = (\nabla_{E}(-\nu))^{\dagger}$, it follows from (\ref{jac_trans_bd_pf_eqn_2}) that for $t > 0$ the Jacobi field $J^{*}(t)$ is given by:  

\begin{equation}
\label{midpoint_jac_eqn}
\displaystyle J^{*}(t) = \left( \ct_{\Kcal}(t) - \ct_{\Kcal}(r_{0}) - \frac{1}{\sn_{\Kcal}(r_{0})} \right)J_{E}(t) + J^{\dagger}(t). \bigskip 
\end{equation}

Because $Exp_{x}^{\Hcal}(-r_{0}\nu) \in M$ for all $x \in M$, $J^{*}(r_{0})$ is tangent to $M$ at $A(x)$.  Because equality holds in Proposition \ref{key_prop}, this implies via Proposition \ref{Jac_trans_bd} that $J^{\dagger}(r_{0}) = 0$, and therefore that $J^{\dagger '}(r_{0}) = (\nabla_{E}(-\nu))^{\dagger} = 0$.  Because $\ct_{\Kcal}(\frac{r_{0}}{2}) - \ct_{\Kcal}(r_{0}) - \frac{1}{\sn_{\Kcal}(r_{0})} = 0$, it follows from (\ref{midpoint_jac_eqn}) that $J^{*}(\frac{r_{0}}{2}) = 0$.  Letting $\xi: M \rightarrow \Hcal$ be the map which sends points $x \in M$ to the midpoint of the geodesic chord $l(x,A(x))$, we then have $d\xi_{x}(E) = J^{*}(\frac{r_{0}}{2}) \equiv 0$, so $\xi$ is a constant map and all geodesic chords from $x$ to $A(x)$, $x \in M$, have the same midpoint $p_{0}$.  This implies that $M$ is contained in the geodesic sphere $S_{(\frac{r_{0}}{2})}(p_{0})$ of radius $\frac{r_{0}}{2}$ about $p_{0}$ in $\Hcal$. \\  

When $M$ is a hypersurface, this implies that $M = S_{(\frac{r_{0}}{2})}(p_{0})$.  In this case, the geodesic ball $B_{(\frac{r_{0}}{2})}(p_{0})$ of radius $\frac{r_{0}}{2}$ about $p_{0}$ has constant curvature $\mathcal{K}$ -- this follows from the fact that for any $p \in B_{(\frac{r_{0}}{2})}(p_{0})$ and any unit tangent vector $\vec{u}$ to $\Hcal$ at $p$, $\vec{u}$ is tangent to a geodesic chord of $M$.  Because equality holds in the Rauch comparison theorem for all Jacobi fields along all geodesic chords of $M$, this implies that any $2$-plane in $T_{p}\Hcal$ containing $\vec{u}$ has sectional curvature $\mathcal{K}$.  When the codimension of $M$ is greater than $1$, it follows from what we have established above that $M$ is the base of a geodesic cone in $\Hcal$ with vertex $p_{0}$, in other words, the union of the geodesic segments $l(p_{0},x)$ from $p_{0}$ to $x \in M$.  We will denote this geodesic cone $C(p_{0},M)$.  It also follows from what we have seen that along each geodesic segment $l(p_{0},x)$, any $2$-plane tangent to $C(p_{0},M)$ which contains $l'$ has sectional curvature $\mathcal{K}$ in $\Hcal$.  For each $x_{0} \in M$, we will likewise write $C(x_{0},M)$ for the geodesic cone in $\Hcal$ with vertex $x_{0}$ and base $M$.  The same conclusion holds for the geodesic segments $l(x_{0},x)$ comprising $C(x_{0},M)$:  any $2$-plane tangent to $C(x_{0},M)$ which contains $l'$ has sectional curvature $\mathcal{K}$ in $\Hcal$.  We will finish proving the characterization of equality by showing that the geodesic cones $C(x_{0},M)$, for all $x_{0} \in M$, coincide, and form an embedded totally geodesic submanifold $\mathcal{D}$ isometric to a ball of radius $\frac{r_{0}}{2}$ in a space of constant curvature $\mathcal{K}$.  This submanifold $\mathcal{D}$ coincides with the geodesic cone $C(p_{0},M)$ with base $M$ and vertex $p_{0}$. \\ 

In order to show that the geodesic cones $C(x_{0},M)$ coincide, we will show that the functions $r, \sigma$ on $M \times M$ are functions of one another and satisfy the same relationships as the corresponding functions on a geodesic sphere of radius $\frac{r_{0}}{2}$ in a space of constant curvature $\mathcal{K}$, which we record in (\ref{K_zero_eqn}) and (\ref{K_neg_eqn}) below. \\ 

For fixed $x \in M$, on $M \setminus \lbrace x, A(x) \rbrace$, we have seen that $grad(r) = \cos(\sigma) \widehat{B}_{2}$ and $grad(\sigma) = \frac{\sin(\sigma)}{\sn_{\Kcal}(r)} \widehat{B}_{2}$.  This implies that along any integral curve $c:[0,a] \rightarrow M$ of $\widehat{B}_{2}$, $r(c(s))$ and $\sigma(c(s))$ satisfy the following system of ordinary differential equations: 

\begin{equation*}
\displaystyle \dot{r} = \cos(\sigma), 
\end{equation*}
\begin{equation}
\label{bpch_pf_diff_eqn}
\displaystyle \dot{\sigma} = \frac{\sin(\sigma)}{\sn_{\Kcal}(r)}. \bigskip 
\end{equation}

By (\ref{path_limiting}), any unit tangent vector to $M$ at $A(x)$ is equal to $c'(0)$ for a smooth path $c:[0,a] \rightarrow M$, such that $\dot{c}(s) = -\widehat{B}_{2}$ for $s \in (0,a]$.  By (\ref{bpch_pf_diff_eqn}), along such a path, $r(c(s))$ and $\sigma(c(s))$ solve the same system of ordinary differential equations as the corresponding functions along the equivalent trajectory in a sphere of radius $\frac{r_{0}}{2}$ in a space of constant curvature $\mathcal{K}$, with the same initial conditions.  Therefore, $r(c(s))$ and $\sigma(c(s))$ have the same values at $c(s)$ as on a geodesic sphere in the constant curvature model space. \\ 

When $\mathcal{K} = 0$, this implies that: 

\begin{equation}
\label{K_zero_eqn}
\displaystyle r = r_{0}\sin(\sigma), \bigskip 
\end{equation}
and when $\mathcal{K} < 0$, this implies:  

\begin{equation*}
\displaystyle \cs_{\Kcal}(\frac{r_{0}}{2}) = \cs_{\Kcal}(\frac{r_{0}}{2})\cs_{\Kcal}(r) + \mathcal{K} \sn_{\Kcal}(\frac{r_{0}}{2})\sn_{\Kcal}(r)\sin(\sigma)  
\end{equation*}
\begin{equation}
\label{K_neg_eqn}
\displaystyle \iff \sinh(\sqrt{|\Kcal|}(\frac{r_{0}}{2}))\sinh(\sqrt{|\Kcal|}r)\sin(\sigma) = \cosh(\sqrt{|\Kcal|}(\frac{r_{0}}{2}))\left(\cosh(\sqrt{|\Kcal|}r) - 1 \right). \bigskip 
\end{equation}

By what we have established above, $M \setminus \lbrace x, A(x) \rbrace$ is foliated by the trajectories of $\widehat{B}_{2}$, so this relationship between $r(x,y)$ and $\sigma(x,y)$ holds for all $y \in M$.  Because $x$ is arbitrary, it holds for all $(x,y) \in M \times M$. \\  

Next, we will show that for each $x_{0}$ in $M$ the geodesic cone $C(x_{0},M)$, and $M$ itself, are contained in the image via the exponential map $Exp_{x_{0}}^{\Hcal}$ of the $(m+1)$-dimensional subspace of $T_{x_{0}}\Hcal$ spanned by $T_{x_{0}}M$ and $\nu$.   To see this, consider a point $y$ in $M$ and the geodesic triangle in $\Hcal$ with vertices $p_{0}, x_{0}$ and $y$.  The length of the segments $\overline{x_{0}p_{0}}$ and $\overline{p_{0}y}$ is $\frac{r_{0}}{2}$ and the length of $\overline{y x_{0}}$ is $r(x_{0},y)$.  Letting $\theta_{p}, \theta_{x}$ and $\theta_{y}$ be the angles of this geodesic triangle at $p_{0}, x_{0}$ and $y$ respectively, we have $\theta_{x}, \theta_{y} \geq \frac{\pi}{2} - \sigma(x_{0},y)$. \\  

If $\mathcal{K} = 0$, the Rauch comparison theorem, (\ref{K_zero_eqn}) and the law of cosines imply: 

\begin{equation*}
\displaystyle \left(\frac{r_{0}}{2}\right)^{2} \geq \left(\frac{r_{0}}{2}\right)^{2} + r_{0}^{2}\sin^{2}(\sigma(x_{0},y)) - 2\left(\frac{r_{0}}{2}\right) r_{0} \sin(\sigma(x_{0},y))\cos(\theta_{x}) 
\end{equation*}
\begin{equation*}
\displaystyle \implies \sin(\sigma(x_{0},y))\cos(\theta_{x}) \geq \sin^{2}(\sigma(x_{0},y)). \bigskip 
\end{equation*}

If $\mathcal{K} < 0$, the Rauch comparison theorem and the hyperbolic law of cosines imply: 

\begin{equation*}
\displaystyle \cs_{\Kcal}(\frac{r_{0}}{2}) \geq \cs_{\Kcal}(\frac{r_{0}}{2}) \cs_{\Kcal}(r(x_{0},y)) + \mathcal{K} \sn_{\Kcal}(\frac{r_{0}}{2})\sn_{\Kcal}(r(x_{0},y))\cos(\theta_{x})  
\end{equation*}
\begin{equation*}
\displaystyle \iff \sinh(\sqrt{|\Kcal|}(\frac{r_{0}}{2}))\sinh(\sqrt{|\Kcal|}r)(x_{0},y))\cos(\theta_{x}) = \cosh(\sqrt{|\Kcal|}(\frac{r_{0}}{2}))\left(\cosh(\sqrt{|\Kcal|}r(x_{0},y)) - 1 \right). \bigskip 
\end{equation*}

In both cases, using (\ref{K_neg_eqn}) when $\mathcal{K} < 0$, this implies that $\theta_{x} = \frac{\pi}{2} - \sigma(x_{0},y)$.  Letting $l(x_{0},y)$ be the geodesic segment from $x_{0}$ to $y$ and $l'(0)$ its initial tangent vector, this implies that $l'(0) = \cos(\sigma) \widehat{E}_{2} + \sin(\sigma)\nu$.  This then implies that $M$, and the geodesic cone $C(x_{0},M)$, are contained in $Exp_{x_{0}}^{\Hcal}(T_{x_{0}}M \oplus span(\nu))$. \\  

To establish that all geodesic cones $C(x_{0},M)$ for $x_{0} \in M$ coincide, let $x$, $y$ and $z$ be three points in $M$, and let $\theta$ be the angle between the geodesic segments $l(x,y)$ and $l(x,z)$ at $x$.  Let $\widetilde{x},\widetilde{y},\widetilde{z}$ be the vertices of a geodesic triangle in the model space $\mathcal{M}_{\Kcal}$ of constant curvature $\mathcal{K}$ with $dist(\widetilde{x},\widetilde{y}) = r(x,y)$, $dist(\widetilde{x},\widetilde{z}) = r(x,z)$, and with angle $\theta$ at $\widetilde{x}$.  By the Rauch comparison theorem, $r(y,z) \geq dist(\widetilde{y},\widetilde{z})$.  On the other hand, because for all $w \in M$, equality holds in the Rauch comparison theorem for all Jacobi fields along $l(x,w)$ and tangent to $C(x,M)$, we can transfer the parametrization in normal coordinates about $\widetilde{x}$ of the geodesic segment $\overline{\widetilde{y},\widetilde{z}}$ from $\mathcal{M}_{\Kcal}$ to $\Hcal$ via $Exp_{x}^{\Hcal}$ to obtain a curve from $y$ to $z$, of length $dist(\widetilde{y},\widetilde{z})$.  This curve must therefore be $l(y,z)$, and by construction it is contained in $C(x,M)$.  This then implies $C(y,M) \subseteq C(x,M)$, and therefore that all geodesic cones $C(x_{0},M)$ for $x_{0}$ in $M$ coincide. \\ 

We will denote the common subset of $\Hcal$ described by $\lbrace C(x,M):x \in M \rbrace$ by $\mathcal{D}$.  It is straightforward to check that $\mathcal{D}$ is an embedded, totally geodesic submanifold of $\Hcal$ with constant curvature $\mathcal{K}$, diffeomorphic to a disk, that it coincides with the cone $C(p_{0},M)$ described above, and that it is a metric ball of radius $\frac{r_{0}}{2}$ about $p_{0}$.  By the Cartan-Ambrose-Hicks theorem, any such manifold $\mathcal{D}$ is isometric to a geodesic ball of radius $\frac{r_{0}}{2}$ in a space $\mathcal{M}_{\Kcal}$ of constant curvature $\mathcal{K}$.  It is also straightforward to check the converse -- that equality holds in Theorem \ref{main_thm_2} for the boundary of any totally geodesic submanifold $\mathcal{D}$ isometric to such a geodesic ball. \end{proof}

For a hypersurface $M$ in $\Hcal$, the following formula for $\int_{\Hcal} w(M,p)^{2} dp$ will be used in the proofs of Theorems \ref{main_thm_1} and \ref{generalization}.  In particular, once we have established this result, Theorem \ref{main_thm_1} will follow from Theorem \ref{main_thm_2}: 

\begin{prop}
\label{hyp_main_thm}
	
Let $f: M^{n-1} \rightarrow \Hcal^{n}$ be an immersion of a closed, oriented hypersurface $M$ into a Cartan-Hadamard manifold $\Hcal$, and let $w(M,p)$ be the winding number of $M$ about $p$ in $\Hcal$.  Then:  
	
\begin{equation}
\label{hyp_main_thm_eqn}
\displaystyle \int\limits_{\Hcal} w(M,p)^{2} \ dVol_{\Hcal} = \frac{(-1)^{(\binom{n}{2} + 1)}}{(n-1)! \Sigma_{n-1}} \int\limits_{S(M)} r \ l^{*}(dI)^{n-1}. \medskip
\end{equation}
\end{prop}

\begin{proof} For a point $p$ in $\Hcal$ and a unit tangent vector $\vec{u}$ to $\Hcal$ at $p$, let $\rho_{\vec{u}} : [0,\infty) \rightarrow \Hcal$ be the geodesic ray determined by $\vec{u}$ as in Definition \ref{average}.  Then for almost all $\vec{u} \in U_{p}\Hcal$, at almost all $p \in \Hcal$: 
	
\begin{equation}
\label{winding_formula}
\displaystyle w(M,p) = \sum\limits_{x \in f^{-1}(\rho_{\vec{u}})} i_{x}(\rho_{\vec{u}},M), \bigskip
\end{equation}

We therefore have:  
	
\begin{equation*}
\displaystyle \int\limits_{\Hcal} w(M,p)^{2} \ dVol_{\Hcal} = \frac{1}{\Sigma_{n-1}} \int\limits_{U(\Hcal)} (\sum\limits_{f^{-1}(\rho_{\vec{u}})} i_{x}(\rho_{\vec{u}},M)) ( \sum\limits_{f^{-1}(\rho_{-\vec{u}})} i_{y}(\rho_{-\vec{u}},M)) \ dVol_{U(\Hcal)}. \bigskip
\end{equation*}
	
For a point $\gamma(t)$ on an oriented geodesic $\gamma$ in $\Hcal$, let $M_{\gamma(t)}^{+} = \lbrace f^{-1}(\gamma(\tau)) : \tau > t \rbrace$ and $M_{\gamma(t)}^{-} = \lbrace f^{-1}(\gamma(\tau)) : \tau < t \rbrace$.  Noting that, for a point $y$ along the geodesic $\gamma_{\vec{u}}$ whose orientation is determined by $\vec{u}$, $i_{y}(\rho_{-\vec{u}},M)$ is equal to $-i_{y}(\gamma_{\vec{u}},M)$, we then have:  
	
\begin{equation}
\label{hyp_main_thm_pf_eqn_2}
\displaystyle \int\limits_{\Hcal} w(M,p)^{2} \ dVol_{\Hcal} = \frac{-1}{\Sigma_{n-1}} \int\limits_{\Gcal} \int\limits_{\gamma} ( \sum\limits_{M_{\gamma(t)}^{+}} i_{x}(\gamma,M) ) ( \sum\limits_{M_{\gamma(t)}^{-}} i_{y}(\gamma,M) ) \ dt \ dVol_{\Gcal}. \bigskip  
\end{equation}  
	
For an oriented geodesic $\gamma$, let $M^{\star}_{\gamma}$ be defined by:  

\begin{equation*}
\displaystyle M^{\star}_{\gamma} = \lbrace (x,y) \in M \times M: f(x), f(y) \in \gamma, \ f(y) = \gamma(t_{2}), f(x) = \gamma(t_{1})\ \text{with $t_{2} > t_{1}$.} \rbrace \bigskip  
\end{equation*}
Then by (\ref{hyp_main_thm_pf_eqn_2}), 
	
\begin{equation}
\displaystyle \int\limits_{\Hcal} w(M,p)^{2} \ dVol_{\Hcal} = \frac{-1}{\Sigma_{n-1}} \int\limits_{\Gcal} \sum\limits_{M^{\star}_{\gamma}} i_{x}(\gamma,M) i_{y}(\gamma,M) r(x,y) dVol_{\Gcal}. \bigskip 
\end{equation}
	
By Corollary \ref{sign_cor}, $i_{x}(\gamma,M) i_{y}(\gamma,M) = sgn(det(dl))_{(x,y)}$, so by a change of variables,   
	
\begin{equation}
\displaystyle \int\limits_{\Hcal} w(M,p)^{2} \ dVol_{\Hcal} = \frac{-1}{\Sigma_{n-1}} \int\limits_{S(M)} r \ l^*(dVol_{\Gcal}) = \frac{(-1)^{(\binom{n}{2} + 1)}}{(n-1)! \Sigma_{n-1}} \int\limits_{S(M)} r \ l^*(dI)^{n-1}.  
\end{equation}
\end{proof}

\begin{proof}[Proof of Theorem \ref{main_thm_1}] Let $f:M^{n-1} \rightarrow \Hcal^{n}$ be as in Theorem \ref{main_thm_1}.  Assuming only that the sectional curvature of $\Hcal$ is non-positive, Theorem \ref{main_thm_1} follows immediately from Theorem \ref{main_thm_2} and Proposition \ref{hyp_main_thm}.  If the sectional curvature of $\Hcal$ is bounded above by $\mathcal{K} < 0$, the stronger result in Part B of Theorem \ref{main_thm_1} follows from Proposition \ref{hyp_main_thm} and (\ref{bpch_pf_eqn_4}) in the proof of Theorem \ref{main_thm_1}.  \end{proof}

It would be interesting to know whether, when $D_{m} \int_{S(M)} r l^{*}(dI)^{m}$ does not have a geometric interpretation in terms of linking or winding numbers, it is nonetheless positive.  The proof of Theorem \ref{main_thm_2} shows that $|D_{m} \int_{S(M)} r l^{*}(dI)^{m}| \leq \iint_{M \times M} \frac{1}{\sn_{\Kcal}^{m-1}(r)} dVol_{M \times M}$ regardless of the sign of $D_{m} \int_{S(M)} r l^{*}(dI)^{m}$.  We end this section with the following result, which shows that this upper bound is finite.  

\begin{prop}
\label{finiteness}
	
Let $f:M^{m} \rightarrow \Hcal^{n}$ be an immersion of a closed, oriented $m$-manifold $M$ into a Cartan-Hadamard manifold $\Hcal$.  For $(x,y) \in M \times M$, let $r(x,y)$ be the distance between $f(x)$ and $f(y)$ in $\Hcal$. \\ 
	
Then $\frac{1}{r^{m-1}} \in L^{1}(M \times M)$. \end{prop}

\begin{proof} Because $r^{-1}(0)$ is a closed subset of the compact $M \times M$, it is enough to show $\iint_{M \times M} \frac{1}{r^{m-1}} \ dVol_{M \times M}$ converges on a neighborhood of $r^{-1}(0)$, and by compactness, it is enough to show this on a neighborhood of each $(x,y) \in r^{-1}(0)$.  Letting $\varDelta$ denote the diagonal in $M \times M$, $r^{-1}(0)$ is $\varDelta$ together with points $(x,y) \in M \times M \setminus \varDelta$ such that $f(x) = f(y)$.  We will address these two cases separately.  It will be helpful to note that although $r$ is not smooth on $M \times M$, $\vec{v}(r)$ is well defined for all tangent vectors $\vec{v}$ to $M \times M$.  At any point $(x,y)$ of $r^{-1}(0)$, $\lbrace \vec{v} \ | \ \vec{v}(r) = 0 \rbrace$ is contained in a subspace of $T_{(x,y)} M \times M$ of dimension at most $m$: along $\varDelta$, this is $T_{(x,x)}\varDelta$ and for $x \not= y$, this space has at most the dimension of $df_{x}(T_{x}M) \cap df_{y}(T_{y}M)$.  Because $f$ is an immersion, there is a neighborhood of $\varDelta$, which we can take to be a tube of radius $\varepsilon$ about $\varDelta$ and denote $\varDelta_{\varepsilon}$, such that $r^{-1}(0) \cap \varDelta_{\varepsilon} = \varDelta$.  Calculating $\iint_{M \times M} \frac{1}{r^{m-1}} \ dVol_{M \times M}$ in Fermi coordinates in such a neighborhood shows that the integral converges in a neighborhood of $\varDelta$.  For points $(x_0,y_0) \in M \times M \setminus \varDelta$ with $f(x_0)=f(y_0)$, let $V_0$ be a coordinate neighborhood of $y_0$ such that $f|_{V_0}$ is injective and $x_0 \notin V_0$.  Let $N_0$ be a system of Fermi coordinates for $\Hcal$ about $f(x_0) = f(y_0)$, defined on the normal disk bundle of radius $r_0$ about $f(V_0)$, and let $U_0$ be a coordinate neighborhood of $x_0$ such that $U_0 \cap V_0 = \emptyset$ and $U_0 \subseteq f^{-1}(N_0)$.  For each $x \in U_0$, there is a unique $\eta(x) \in V_0$ such that $f(\eta(x))$ is nearest to $f(x)$.  The map $\eta: U_0 \rightarrow V_0$ is smooth, and $\lbrace (x,\eta(x)) | x \in U_0 \rbrace$ is a smooth $m$-dimensional submanifold of $M \times M$ which contains $r^{-1}(0)$ in a neighborhood of $(x_0,y_0)$.  Calculating $\iint_{M \times M} \frac{1}{r^{m-1}} \ dVol_{M \times M}$ in Fermi coordinates about $\lbrace (x,\eta(x)) \rbrace$ again shows that the integral converges in a neighborhood of $(x_0,y_0)$. \end{proof}


\section{Hypersurfaces and Isoperimetric Inequalities}
\label{isoperimetric_inequalities} 


In this section, we will prove Theorems \ref{generalization}, \ref{yau_inequality} and \ref{Howard_cor}.  We will first prove Theorem \ref{yau_inequality} and then show that Theorem \ref{Howard_cor} follows from the $2$-dimensional cases of this result and Theorem \ref{main_thm_1}.  We will then prove Theorem \ref{generalization} and discuss ways that this result may be strengthened and extended in future work. \\ 

In proving Theorems \ref{yau_inequality} and \ref{Howard_cor}, we will write $\rho_{\vec{u}}$ for the geodesic ray in $\Hcal$ determined by a unit tangent vector $\vec{u}$ to $\Hcal$, as in Definition \ref{average} and Proposition \ref{hyp_main_thm}.  We will view $\rho_{\vec{u}}$ as both a geometric object and a mapping $\rho_{\vec{u}}:(0,\infty) \rightarrow \Hcal$, in which case $\rho_{\vec{u}}$ will be parametrized with unit speed.  As above, we will write $\zeta: U(\Hcal) \times \R \rightarrow U(\Hcal)$ for the geodesic flow of $\Hcal$ so, for example, $\zeta^{t}(\vec{u}) = \rho_{\vec{u}}'(t)$ for $t > 0$.  For $t > 0$, we will write $\rootg_{\vec{u}}(t)$ for the volume element in normal coordinates at $t\vec{u}$.  In $\R^{n}$, $\rootg_{\vec{u}}(t) = t^{n-1}$, and in an $n$-dimensional hyperbolic space with sectional curvature $\mathcal{K}$, $\rootg_{\vec{u}}(t) = \sn_{\Kcal}^{n-1}(t)$.  It follows from \cite[Lemma 5]{Yau} that:  

\begin{equation}
\label{yau_symmetry}
\displaystyle \sqrt{g}_{\vec{u}}(t) = \sqrt{g}_{-\zeta^{t}(\vec{u})}(t) = \sqrt{g}_{-\rho_{\vec{u}}'(t)}(t). \bigskip 
\end{equation}

The proof of the Bishop-G\"unther-Gromov volume comparison theorem implies that for any Cartan-Hadamard manifold $\Hcal^{n}$ with sectional curvature bounded above by $\mathcal{K} \leq 0$, for all $\vec{u} \in U(\Hcal)$ and $t > 0$, $\frac{\rootg_{\vec{u}}(t)}{\sn_{\Kcal}^{n-1}(t)} \geq 1$, with equality if and only if all $2$-planes tangent to $\rho_{\vec{u}}$ on the interval $[0,t]$ have sectional curvature $\mathcal{K}$.  We refer to \cite{BC,Gr} for the proof of the Bishop-G\"unther-Gromov volume comparison theorem and more background about this inequality.  We will use the following corollary of this fact in the proof of Theorem \ref{yau_inequality}: 

\begin{lemma}
\label{bg_lemma}
Let $\gamma:(-\infty,\infty) \rightarrow \Hcal^{n}$ be a geodesic in an $n$-dimensional Cartan-Hadamard manifold with curvature bounded above by $\mathcal{K} \leq 0$.  Let $0 < s < r$ and $0 < t < r - s$.  Then: 

\begin{equation}
\label{bg_lemma_eqn_1}
\displaystyle \frac{\sqrt{g}_{\gamma'(0)}(r)}{\sqrt{g}_{\gamma'(s)}(t)} \geq \frac{\sn_{\Kcal}^{n-1}(r)}{\sn_{\Kcal}^{n-1}(t)}, \medskip 
\end{equation}
with equality if and only if all $2$-planes tangent to $\gamma([0,r])$ have sectional curvature $\mathcal{K}$. 
\end{lemma}

\begin{proof} To begin, 

\begin{equation}
\displaystyle \frac{\sqrt{g}_{\gamma'(0)}(r)}{\sqrt{g}_{\gamma'(s)}(t)} = \left( \frac{\sqrt{g}_{\gamma'(0)}(r)}{\sqrt{g}_{\gamma'(s)}(r-s)} \right) \left( \frac{\sqrt{g}_{\gamma'(s)}(r-s)}{\sqrt{g}_{\gamma'(s)}(t)} \right).  \bigskip 
\end{equation}

By (\ref{yau_symmetry}), $\frac{\sqrt{g}_{\gamma'(0)}(r)}{\sqrt{g}_{\gamma'(s)}(r-s)} = \frac{\sqrt{g}_{-\gamma'(r)}(r)}{\sqrt{g}_{-\gamma'(r)}(r-s)}$.  By the proof of the Bishop-G\"unther-Gromov inequality, $\frac{\sqrt{g}_{-\gamma'(r)}(r)}{\sqrt{g}_{-\gamma'(r)}(r-s)} \geq \frac{r^{n-1}}{(r-s)^{n-1}}$ and $\frac{\sqrt{g}_{\gamma'(s)}(r-s)}{\sqrt{g}_{\gamma'(s)}(t)} \geq \frac{(r-s)^{n-1}}{t^{n-1}}$, with equality only if the sectional curvature of all $2$-planes tangent to $\gamma([0,r])$ is $\mathcal{K}$. \end{proof}

We will write $V_{\Kcal}^{n}(s)$ for $\int_{0}^{s} \sn_{\Kcal}^{n-1}(t) dt$, so that the volume of a ball of radius $s$ in an $n$-dimensional hyperbolic space of curvature $\mathcal{K}$ is $\Sigma_{n-1} V_{\Kcal}^{n}(s)$, and we will write $W_{\Kcal}^{n}(r)$ for $\int_{0}^{r}V_{\Kcal}^{n}(s) ds$.  For $n=2$ we have $V_{\Kcal}^{2}(s) = \frac{1}{|\Kcal|}(\cs_{\Kcal}(s) - 1)$, $W_{\Kcal}^{2}(r) = \frac{1}{|\Kcal|}(\sn_{\Kcal}(r) - r)$ and for $n = 3$ we have $V_{\Kcal}^{3}(s) = \frac{1}{4|\Kcal|}(\sn_{\Kcal}(2s) - 2s)$, $W_{\Kcal}^{3}(r) = \frac{1}{4|\Kcal|}(\sn_{\Kcal}^{2}(r) - r^{2})$.  An elementary calculation shows that for $n \geq 4$, $W_{\Kcal}^{n}(r)$ can be described inductively by:  

\begin{equation}
\label{int_inductive_formula}
\displaystyle W_{\Kcal}^{n}(r) = \frac{1}{(n-1)|\Kcal|} \left( \frac{\sn_{\Kcal}^{n-1}(r)}{(n-1)} - (n-2) W_{\Kcal}^{n-2}(r) \right). \bigskip 
\end{equation}

We will write $\nabla r$ for the gradient of the chordal distance function $r$ on $M \times M$.  Note that this is different from $grad(r)$ as in the proof of Theorem \ref{main_thm_2}:  in the proof of Theorem \ref{main_thm_2} we fix $x \in M$ and consider $r$ as a function on $M \setminus f^{-1}(x)$.  Here we take the gradient of $r$ as a function on the Riemannian manifold $M \times M$.  We define: 

\begin{equation*}
\displaystyle \Psi_{\Kcal}^{2}(r,|\nabla r|) = |\nabla r| + (1 - |\nabla r|)\frac{r}{\sn_{\Kcal}(r)}, 
\end{equation*}
\begin{equation*}
\displaystyle \Psi_{\Kcal}^{3}(r,|\nabla r|) = |\nabla r| + (1 - |\nabla r|)\frac{r^{2}}{\sn_{\Kcal}^{2}(r)}, 
\end{equation*}
\begin{equation}
\label{yau_inequality_function}
\displaystyle \Psi_{\Kcal}^{n}(r,|\nabla r|) = |\nabla r| + \frac{(1 - |\nabla r|)(n-1)(n-2) W_{\Kcal}^{n-2}(r)}{\sn_{\Kcal}^{n-1}(r)} \bigskip  
\end{equation}
for $n \geq 4$.  Because $0 \leq |\nabla r| \leq 1$, $\Psi_{\Kcal}^{n}(r,|\nabla r|) > 0$ for all $(x,y) \in M \times M$ with $r > 0$.  

\begin{proof}[Proof of Theorem \ref{yau_inequality}] For almost all $\vec{u} \in U_{p}\Hcal$, at almost $p \in \Hcal$, 

\begin{equation}
\label{int_sq_pf_eq_1}
\displaystyle |w(M,p)| \leq \chi(f^{-1}(\rho_{\vec{u}})). \bigskip 
\end{equation}

This implies that for almost all $p_0$ in $\Hcal$, 

\begin{equation*}
\displaystyle \int\limits_{\Hcal} |w(M,p)| dVol_{\Hcal} \leq \int\limits_{U_{p_0}\Hcal}\int\limits_{0}^{\infty} \sqrt{g}_{\vec{u}}(t) \ \chi(f^{-1}(\rho_{\zeta^{t}(\vec{u})})) \ dt \ dVol_{U_{p_0}(\Hcal)}. \bigskip
\end{equation*}

Because (\ref{int_sq_pf_eq_1}) also holds for almost all $-\vec{u}$, that is, $|w(M,p_0)| \leq \chi(f^{-1}(\rho_{\vec{u}}))$ and $|w(M,p_0)| \leq \chi(f^{-1}(\rho_{-\vec{u}}))$,   

\begin{equation*}
\displaystyle \left( \int\limits_{\Hcal} |w(M,p)| dVol_{\Hcal} \right)^{2} \leq \int\limits_{U(\Hcal)} \chi(f^{-1}(\rho_{-\vec{u}})) \int\limits_{0}^{\infty} \sqrt{g}_{\vec{u}}(t) \  \chi(f^{-1}(\rho_{\zeta^{t}(\vec{u})})) dt dVol_{U(\Hcal)}  
\end{equation*}

\begin{equation}
\label{int_sq_pf_eq_3}
\displaystyle = \int\limits_{U(\Hcal)} \chi(f^{-1}(\rho_{-\vec{u}})) \left( \sum\limits_{y \in f^{-1}(\rho_{\vec{u}})} \int\limits_{0}^{r(\pi(\vec{u}),y)} \sqrt{g}_{\vec{u}}(t) \ dt \right) dVol_{U(\Hcal)}. \bigskip
\end{equation}

Given an oriented geodesic $\gamma$ in $\Hcal$, let $\gamma: (-\infty,\infty) \rightarrow \Hcal$ be a unit-speed parametrization, and for $s \in (-\infty,\infty)$, let $\gamma_{s}^{-} = \gamma(-\infty,s)$ and $\gamma_{s}^{+} = \gamma(s,\infty)$ be the geodesic rays determined by $\gamma(s)$.  Rewriting the right-hand side of (\ref{int_sq_pf_eq_3}) as an integral over $\Gcal$, we have: 

\begin{equation*}
\label{int_sq_pf_eq_5}
\displaystyle \left( \int\limits_{\Hcal} |w(M,p)| dVol_{\Hcal} \right)^{2} \leq \int\limits_{\Gcal}\int\limits_{-\infty}^{\infty} \chi( f^{-1}(\gamma_{s}^{-})) \left( \sum\limits_{y \in f^{-1}(\gamma_{s}^{+})} \int\limits_{0}^{r(\gamma(s),y)} \sqrt{g}_{\gamma'(s)}(t) \ dt \right) ds \ dVol_{\Gcal}. \bigskip 
\end{equation*}

By Proposition \ref{crofton_santalo}, the integral over $\Gcal$ in the right-hand side of this inequality can be rewritten as an integral over $\Ucal(f)$.  Letting $\sigma_{1} \in [0,\frac{\pi}{2}]$ be the angle that a unit vector $\vec{u}$ in $U_{x}\Hcal$ makes with $T_{x}M$, the factor $|\langle \vec{u}, \nu \rangle|$ introduced into the integrand in Proposition \ref{crofton_santalo} is equal to $\sin(\sigma_{1})$.  Therefore: 

\begin{equation*}
\displaystyle \left( \int\limits_{\Hcal} |w(M,p)| dVol_{\Hcal} \right)^{2} \leq \int\limits_{M} \int\limits_{U_{x}\Hcal} \sin(\sigma_{1}) \int\limits_{0}^{\infty} \sum\limits_{y \in f^{-1}(\rho_{\zeta^{s}(\vec{u})})} \int\limits_{0}^{r(\rho_{\vec{u}}(s),y)} \sqrt{g}_{\zeta^{s}(\vec{u})}(t) \ dt \ ds \ d\vec{u} \ dVol_{M}  
\end{equation*}

\begin{equation}
\label{int_sq_pf_eq_6}
\displaystyle = \int\limits_{M} \int\limits_{U_{x}\Hcal} \sin(\sigma_{1}) \sum\limits_{y \in f^{-1}(\rho_{\vec{u}})} \int\limits_{0}^{r(x,y)} \int\limits_{0}^{r(x,y)-s} \sqrt{g}_{\zeta^{s}(\vec{u})}(t) \ dt \ ds \ d\vec{u} \ dVol_{M}. \bigskip 
\end{equation}

By Sard's theorem, $\rho_{\vec{u}}$ meets $M$ transversely in finitely many points (or none) for almost all $\vec{u} \in U_{x}\Hcal$.  We can therefore rewrite the integral over $U_{x}\Hcal$ in (\ref{int_sq_pf_eq_6}) as an integral over $M$.  Let $l(x,y)$ be the geodesic chord through $x$ and $y \in M$ as above, parametrized so that $l(0) = x$ and $l(r) = y$, and let $\sigma_{2} \in [0,\frac{\pi}{2}]$ be the angle that $l(x,y)$ makes with $M$ at $y$.  In this notation, $\sqrt{g}_{\zeta^{s}(\vec{u})}(t) = \sqrt{g}_{l'(s)}(t)$, and the change-of-variables from integration over $\vec{u} \in U_{x}\Hcal$ to integration over $y \in M$ introduces the factor $\frac{\sin(\sigma_{2})}{\sqrt{g}_{l'(0)}(r)}$ into the integrand in (\ref{int_sq_pf_eq_6}).  We then have:  

\begin{equation}
\label{int_sq_pf_eq_7}
\displaystyle \left( \int\limits_{\Hcal} |w(M,p)| dVol_{\Hcal} \right)^{2} \leq \iint\limits_{M \times M} \frac{\sin(\sigma_{1})\sin(\sigma_{2})}{\sqrt{g}_{l'(0)}(r)} \int\limits_{0}^{r(x,y)} \int\limits_{0}^{r(x,y)-s} \sqrt{g}_{l'(s)}(t) \ dt ds dVol_{M \times M}. \bigskip 
\end{equation}

Lemma \ref{bg_lemma} implies that $\frac{\sqrt{g}_{l'(s)}(t)}{\sqrt{g}_{l'(0)}(r)} \leq \frac{\sn_{\Kcal}^{n-1}(t)}{\sn_{\Kcal}^{n-1}(r)} = \frac{\sinh^{n-1}(\sqrt{|\Kcal|}t)}{\sinh^{n-1}(\sqrt{|\Kcal|}r)}$, so (\ref{int_sq_pf_eq_7}) implies: 

\begin{equation*}
\displaystyle \left( \int\limits_{\Hcal} |w(M,p)| dVol_{\Hcal} \right)^{2} \leq \iint\limits_{M \times M} \frac{\sin(\sigma_{1})\sin(\sigma_{2})}{\sn_{\Kcal}^{n-1}(r)} \int\limits_{0}^{r(x,y)} \int\limits_{0}^{r(x,y)-s} \sn_{\Kcal}^{n-1}(t) \ dt \ ds \ dVol_{M \times M} 
\end{equation*}

\begin{equation*}
\displaystyle = \iint\limits_{M \times M} \frac{\sin(\sigma_{1})\sin(\sigma_{2})}{\sn_{\Kcal}^{n-1}(r)} \int\limits_{0}^{r} V_{\Kcal}^{n}(r-s) \ ds \ dVol_{M \times M} = \iint\limits_{M \times M}  \frac{\sin(\sigma_{1})\sin(\sigma_{2})W_{\Kcal}^{n}(r)}{\sn_{\Kcal}^{n-1}(r)} \ dVol_{M \times M}
\end{equation*}

\begin{equation*}
\displaystyle \leq \iint\limits_{M \times M} \frac{( 1 - \cos(\sigma_{1})\cos(\sigma_{2}))W_{\Kcal}^{n}(r)}{\sn_{\Kcal}^{n-1}(r)} \ dVol_{M \times M} 
\end{equation*}

\begin{equation}
\label{int_sq_pf_eq_9}
\displaystyle = \iint\limits_{M \times M} \left( 1 - |\nabla r|\right)\frac{W_{\Kcal}^{n}(r)}{\sn_{\Kcal}^{n-1}(r)} \ dVol_{M \times M}. \bigskip
\end{equation}

For $n=2,3$, the inequality in Theorem \ref{yau_inequality} follows immediately from (\ref{yau_inequality_function}) and (\ref{int_sq_pf_eq_9}).  For $n \geq 4$, by (\ref{int_sq_pf_eq_9}) and the identity (\ref{int_inductive_formula}) for $W_{\Kcal}^{n}(r)$,    

\begin{equation*}
\displaystyle \left((n-1)\sqrt{|\Kcal|} \int\limits_{\Hcal} |w(M,p)| dVol_{\Hcal} \right)^{2} \leq \iint\limits_{M \times M} (1 - |\nabla r|) \left( 1 - \frac{(n-1)(n-2) W_{\Kcal}^{n-2}(r)}{\sn_{\Kcal}^{n-1}(r)} \right) 
\end{equation*}

\begin{equation}
\label{int_sq_pf_eq_10}
\displaystyle = Vol(M)^{2} - \iint\limits_{M \times M} \left( |\nabla r| + \frac{(1 - |\nabla r|)(n-1)(n-2) W_{\Kcal}^{n-2}(r)}{\sn_{\Kcal}^{n-1}(r)} \right), \bigskip  
\end{equation}
which is the inequality in Theorem \ref{yau_inequality}. \\ 

Equality requires that equality holds in (\ref{int_sq_pf_eq_3}).  This implies that for almost all $p \in \Hcal$ with $w(M,p) \neq 0$, for almost all $\vec{u} \in U_{p}\Hcal$, $w(M,p) = \chi(f^{-1}(\rho_{\vec{u}})) = \chi(f^{-1}(\rho_{-\vec{u}}))$.  This implies that the geodesic $\gamma$ to which $\vec{u}$ and $-\vec{u}$ are tangent can intersect $M$ in at most two distinct points in $\Hcal$.  Therefore, almost all geodesics in $\Hcal$ intersect $f(M)$ at most twice, and $f(M)$ is the boundary of a convex domain $\Omega$ in $\Hcal$.  Equality also requires equality in (\ref{int_sq_pf_eq_9}).  This implies that all $2$-planes tangent to geodesic chords of $f(M)$ have sectional curvature equal to $\mathcal{K}$, and therefore that $\Omega$ has constant curvature, as in the proof of Theorem \ref{main_thm_2} in the codimension $1$ case.  Finally, equality requires that $\sin(\sigma_{1})\sin(\sigma_{2}) \equiv 1 - \cos(\sigma_{1})\cos(\sigma_{2})$, which implies that all geodesic chords of $f(M)$ meet $f(M)$ at the same angle at both geometrically distinct intersection points.  This implies that the domain $\Omega$ is isometric to a ball in an $n$-dimensional hyperbolic space with constant curvature $\mathcal{K}$.  \end{proof}

Assuming only that the curvature of $\Hcal$ is non-positive, using $\frac{\sqrt{g}_{l'(s)}(t)}{\sqrt{g}_{l'(0)}(r)} \leq \frac{t^{n-1}}{r^{n-1}}$ in (\ref{int_sq_pf_eq_7}), the proof of Theorem \ref{yau_inequality} implies:  

\begin{prop}
\label{sq_int_cor}
Let $f:M^{n-1} \rightarrow \Hcal^{n}$ be an immersion of a closed, oriented hypersurface in a Cartan-Hadamard manifold.  Then:  

\begin{equation}
\label{sq_int_thm_eqn_1}
\displaystyle \left( \int\limits_{\Hcal} |w(M,p)| dVol_{\Hcal} \right)^{2} \leq \frac{1}{n(n+1)} \iint\limits_{M \times M} (1-|\nabla r|)r^{2} \ dVol_{M \times M}. \medskip 	
\end{equation}

Equality holds if and only if $M$ is the boundary, possibly with multiplicity, of a domain $\Omega$ isometric to a ball in $\R^{n}$. \end{prop}

Theorem \ref{Howard_cor} follows from the proofs of Theorems \ref{main_thm_1} and \ref{yau_inequality}: 

\begin{proof}[Proof of Theorem \ref{Howard_cor}]
	
Let $\Hcal^{2}$ be a Cartan-Hadamard surface and $M$ a closed curve in $\Hcal$ of length $L$.  Assuming only that $\Hcal$ has non-positive curvature, Theorem \ref{Howard_cor} is Theorem \ref{main_thm_1}.  Assuming that the curvature of $\Hcal$ is bounded above by $\mathcal{K} < 0$, by (\ref{int_sq_pf_eq_9}) in the proof of Theorem \ref{yau_inequality}, 

\begin{equation*}
\displaystyle |\mathcal{K}| \left( \int\limits_{\Hcal} |w(M,p)| dVol_{\Hcal} \right)^{2} \leq |\mathcal{K}| \iint\limits_{M \times M} \frac{\sin(\sigma_{1})\sin(\sigma_{2})W_{\Kcal}^{2}(r)}{\sn_{\Kcal}(r)} dVol_{M \times M} 
\end{equation*}

\begin{equation}
\label{Howard_cor_pf_eqn_1}
\displaystyle = \iint\limits_{M \times M}  \left( \frac{\sn_{\Kcal}(r) - r}{\sn_{\Kcal}(r)} \right) \sin(\sigma_{1})\sin(\sigma_{2}) dVol_{M \times M}. \bigskip
\end{equation}
	
By (\ref{main_thm_1_B_eqn}), 

\begin{equation}
\label{Howard_cor_pf_eqn_2}
\displaystyle 4\pi \int\limits_{\Hcal} w(M,p)^{2} dA_{\Hcal} \leq L^{2} - \iint\limits_{M \times M} \left( \frac{\sn_{\Kcal}(r) - r}{\sn_{\Kcal}(r)} \right) \sin(\sigma_{1})\sin(\sigma_{2}) dVol_{M \times M}. \bigskip 
\end{equation}

Adding (\ref{Howard_cor_pf_eqn_1}) and (\ref{Howard_cor_pf_eqn_2}) gives the inequality in Theorem \ref{Howard_cor}.  Equality requires equality in Theorem \ref{main_thm_1} for $M$ in $\Hcal$, which has the same equality conditions as Theorem \ref{Howard_cor}. \end{proof}

We will prove Theorem \ref{generalization} using the some of the results from Section \ref{secants} that we used to prove Theorems \ref{main_thm_1} and \ref{main_thm_2}: 

\begin{proof}[Proof of Theorem \ref{generalization}] For $n=2$, Theorem \ref{generalization} is Theorem \ref{main_thm_1}.  For $n \geq 3$, we will write $\overline{S(M)}$ for the secant space of $M$ with the opposite orientation to the one given in Definition \ref{secant_space} and $\overline{S(M)}^{+}$ for the subset of $\overline{S(M)}$ on which $l:\overline{S(M)} \rightarrow \Gcal$ has a positive Jacobian.  By Corollary \ref{sign_cor}, $\overline{S(M)}^{+}$ coincides, up to a set of measure $0$, with the points $(x,y) \in M \times M$ for which $i_{x}(l,M) \neq i_{y}(l,M)$.  In particular, $(x,y) \in \overline{S(M)}^{+}$ precisely if $(y,x) \in \overline{S(M)}^{+}$.  By Theorem \ref{hyp_main_thm} and Definition \ref{G_measure} of the measure on $\Gcal$, 

\begin{equation*}
\displaystyle \Sigma_{n-1}\int\limits_{\Hcal} w(M,p)^{2} \ dVol_{\Hcal} = \int\limits_{\overline{S(M)}} r \ l^{*}(dVol_{\Gcal}) \leq \int\limits_{\overline{S(M)}^{+}} r \ l^{*}(dVol_{\Gcal})  
\end{equation*}

\begin{equation}
\label{generalization_pf_eqn_1}
\displaystyle = \int\limits_{\overline{S(M)}^{+}} \frac{r(\sin(\sigma_{1})\sin(\sigma_{2}))^{\frac{1}{n-1}}}{(\sin(\sigma_{1})\sin(\sigma_{2}))^{\frac{1}{n-1}}} \ l^{*}(dVol_{\Gcal}).
\end{equation}

By H\"older's inequality, $\Sigma_{n-1}\int\limits_{\Hcal} w(M,p)^{2} \ dVol_{\Hcal}$ is then bounded above by:  

\begin{equation}
\label{generalization_pf_eqn_2}
\displaystyle \left( \int\limits_{\overline{S(M)}^{+}} \frac{r^{n-1}}{\sin(\sigma_{1})\sin(\sigma_{2})} l^{*}(dVol_{\Gcal}) \right)^{\frac{1}{n-1}} \left( \int\limits_{\overline{S(M)}^{+}} (\sin(\sigma_{1})\sin(\sigma_{2}))^{\frac{1}{n-2}} l^{*}(dVol_{\Gcal}) \right)^{\frac{n-2}{n-1}}. \bigskip 
\end{equation}

By Proposition \ref{key_prop}, 

\begin{equation}
\label{generalization_pf_eqn_3}
\displaystyle \int\limits_{\overline{S(M)}^{+}} \frac{r^{n-1}}{\sin(\sigma_{1})\sin(\sigma_{2})} l^{*}(dVol_{\Gcal}) \leq Vol(\overline{S(M)}^{+}) \leq Vol(M)^{2}. \bigskip 
\end{equation}

By the arithmetic-geometric mean inequality, 

\begin{equation}
\label{generalization_pf_eqn_4}
\displaystyle \int\limits_{\overline{S(M)}^{+}} (\sin(\sigma_{1})\sin(\sigma_{2}))^{\frac{1}{n-2}} l^{*}(dVol_{\Gcal}) \leq \int\limits_{\overline{S(M)}^{+}} \frac{\sin^{\frac{2}{n-2}}(\sigma_{1}) + \sin^{\frac{2}{n-2}}(\sigma_{2})}{2} l^{*}(dVol_{\Gcal}). \bigskip 
\end{equation}

Using the fact that $\overline{S(M)}^{+}$ is invariant under the exchange of factors $(x,y) \mapsto (y,x) $, that this mapping preserves the measure induced by $l^{*}(dVol_{\Gcal})$ and that $\sigma_{2}(x,y) = \sigma_{1}(y,x)$, we have: 

\begin{equation*}
\displaystyle \int\limits_{\overline{S(M)}^{+}} \frac{\sin^{\frac{2}{n-2}}(\sigma_{1}) + \sin^{\frac{2}{n-2}}(\sigma_{2})}{2} l^{*}(dVol_{\Gcal}) = \int\limits_{\overline{S(M)}^{+}} \sin^{\frac{2}{n-2}}(\sigma_{1}) l^{*}(dVol_{\Gcal}) 
\end{equation*}

\begin{equation}
\label{generalization_pf_eqn_5}
\displaystyle = \int\limits_{\Gcal} \left( \sum\limits_{l^{-1}(\gamma) \cap \overline{S(M)}^{+}} \sin^{\frac{2}{n-2}}(\sigma_{1}) \right) dVol_{\Gcal}. \bigskip 
\end{equation}

For a pair of points $x,y$ in $M$ with $f(x), f(y) \in \gamma$, we will write $f(x) < f(y)$ to indicate that $f(y)$ maps to a point ahead of $f(x)$ according to the orientation of $\gamma$.  Letting $\sigma \in [0,\frac{\pi}{2}]$ for the angle that $\gamma$ makes with $M$ at $f(x)$, (\ref{generalization_pf_eqn_5}) is then equal to:  

\begin{equation}
\label{generalization_pf_eqn_6}
\displaystyle \int\limits_{\Gcal} \sum\limits_{f(x) \in \gamma} \sin^{\frac{2}{n-2}}(\sigma) \chi( \lbrace y \in M \ | \ f(y) > f(x), i_{x}(\gamma,M) \neq i_{y}(\gamma,M) \rbrace ) dVol_{\Gcal}. \bigskip  
\end{equation}

By Proposition \ref{crofton_santalo}, this is equal to: 

\begin{equation}
\label{generalization_pf_eqn_7}
\displaystyle \int\limits_{\Ucal(f)} |\sin(\sigma)|^{\frac{n}{n-2}} \chi( \lbrace y \in M \ | \ f(y) > f(x), i_{x}(\gamma,M) \neq i_{y}(\gamma,M) \rbrace ) dVol_{\Ucal}, \bigskip 
\end{equation} 
where $x$ in the above stands for the point in $M$ at which $\vec{u} \in \Ucal(f)$ is based.  This is equal to $\frac{1}{c_{n}}\int_{M} \mu(x) dVol_{M}$, where $c_{n}$ is the constant in (\ref{mass_function_constant}).  The inequality in Theorem \ref{generalization} then follows from (\ref{generalization_pf_eqn_2}), (\ref{generalization_pf_eqn_3}) and the upper bound in (\ref{generalization_pf_eqn_7}) for $\int_{\overline{S(M)}^{+}} (\sin(\sigma_{1})\sin(\sigma_{2}))^{\frac{1}{n-2}} l^{*}(dVol_{\Gcal})$. \\

Equality requires equality in (\ref{generalization_pf_eqn_1}), which implies that for almost all geodesics $\gamma \in \Gcal$, for each $(x,y) \in S(M)$ with $l(x,y) = \gamma$, $i_{x}(\gamma,M) \neq i_{y}(\gamma,M)$.  This implies that almost all geodesics in the image of $l$ have exactly two geometrically distinct points of intersection with $f(M)$ in $\Hcal$, and thus that the image of $f$ is a convex hypersurface in $\Hcal$, possibly with multiplicity.  Equality requires equality in (\ref{generalization_pf_eqn_3}), which requires equality in Proposition \ref{key_prop} and thus that all $2$-planes tangent to geodesic chords of $M$ have sectional curvature $0$.  As in the proof of the codimension-$1$ case of Theorem \ref{main_thm_2}, this implies the domain $\Omega$ bounded by $f(M)$ is flat.  Equality requires equality in (\ref{generalization_pf_eqn_4}), which implies that $\sin(\sigma_{1}) \equiv \sin(\sigma_{2})$.  Finally, equality requires equality in H\"older's inequality in (\ref{generalization_pf_eqn_2}) which, together with the fact that $\sin(\sigma_{1}) \equiv \sin(\sigma_{2})$, implies there is a constant $K_{f}$ such that: 

\begin{equation}
\label{generalization_pf_eqn_8}
\displaystyle \frac{r^{n-1}}{\sin(\sigma_{1})^{2}} = K_{f}(\sin(\sigma_{1}))^{\frac{2}{n-2}}, \bigskip 
\end{equation}
and therefore that $r^{n-1} = K_{f}\sin(\sigma_{1})^{\frac{2n-2}{n-2}}$.  Together with the conditions above, this is only possible for $n=4$, for domains isometric to a ball in $\R^{4}$. \end{proof} 

\begin{remark}
\label{croke_remark}

We believe it would be natural to try to prove a strengthened version of Theorem \ref{generalization} in which the average $\Acal(f)$ over $M$ is replaced by the average over $\Hcal$ of the absolute value of the winding number of $M$, or another similar invariant which involves averaging over the domains enclosed by $f(M)$, rather than over $M$ itself.  Such a result may give a stronger inequality, which coincides with Croke's isoperimetric inequality for embedded hypersurfaces.  If $D_{m} \int_{S(M)} r l^{*}(dI)^{m}$ is positive for all $f:M^{m} \rightarrow \Hcal^{n}$, as discussed after the proof of Proposition \ref{hyp_main_thm}, one could replace the condition defining $\overline{S(M)}^{+}$ in the proof of Theorem \ref{generalization} with the condition $\langle D_{m} l^{*}(dI)^{m}, dVol_{S(M)} \rangle > 0$.  In that case, the only point in the proof of Theorem \ref{generalization} which depends on the codimension $1$ condition of the immersion $f:M \rightarrow \Hcal$ is the identity for $\int_{\overline{S(M)}^{+}} \sin^{\frac{2}{n-2}}(\sigma_{1}) l^{*}(dVol_{\Gcal}) $ based on Proposition \ref{crofton_santalo}.  It would be interesting to know whether this invariant has a natural geometric interpretation for submanifolds of Cartan-Hadamard manifolds of higher codimension. \end{remark}


\end{document}